\documentclass{amsart}
\usepackage{tikz}
\usepackage{amsthm}

\newtheorem{thm}{Theorem}
\newtheorem{lem}{Lemma}
\newtheorem{cor}{Corollary}
\newtheorem{defi}{Definition}
\newtheorem{con}{Conjecture}

\title{Spaces of random plane triangulations and the density of states}
\author{Nathan Hannon}

\begin{document}
\maketitle

\section{Introduction}

An important tool for studying tilings is the tiling space.  \cite{sadun} The set of all tilings of $\mathbb{R}^n$ forms a topological space under a metric defined in such a way that two tilings are close if and only if, after a small translation, they agree on a large ball around the origin.  The tiling space of a tiling $T$, also called the hull of $T$, is the closure in this space of the set of all translates of $T$.  Another, perhaps more intuitive, characterization of the hull of $T$ is the set of all tilings $T'$ such that every finite patch of $T'$ can be found, up to translation, in $T$.  There are variations on this construction that also take rotations into account.  The spectral properties, particularly the integrated density of states (IDS), of operators on tiling spaces are of interest as they are related to physical properties of solids modeled by those tilings.  These properties have been determined in several cases, such as by Julien and Savinien \cite{julien}.

In a more general context in which a metric space is equipped with a group action (which includes tiling spaces with the action of translation), Lenz and Veseli\'{c} \cite{lenz2007} determined that the IDS of a class of operators can be approximated uniformly by analogues constructed on finite sets, and that jumps of the IDS correspond to compactly supported eigenfunctions of those operators.  In a different but related setting, Beckus and Pogorzelski \cite{beckus2020} proved that the density of states of a random operator on a Delone dynamical system is continuous with respect to the system (under appropriately defined topologies).

Other results have dealt with groupoid structures.  For example, Lenz et al.\,\cite{lenz2002} constructed a von Neumann algebra, trace, and density of states in a setting involving random operators on a groupoid.  Additionally, Beckus et al. \cite{beckus2018} proved that the spectra of certain operators on subsets of a groupoid are continuous, with respect to suitable topologies, as a function of the subset.  Gap-labeling conjectures for some cases have been proven by Benameur and Mathai \cite{benameur2018, benameur2020} and by Kaminker and Putnam \cite{kaminker2003}.

We want to generalize the tiling space construction to tiling-like structures that do not live in $\mathbb{R}^n$ - in this case, random triangulations of 2-manifolds, although the construction could easily work with any sufficiently well-behaved cell complex.  Although most of the triangulations that we will consider are homeomorphic to $\mathbb{R}^2$, they have no notion of translation or any useful group action, but can be given a groupoid structure.  Our goal is to prove results analogous to Lenz and Veseli\'{c} \cite{lenz2007} for these spaces.

Triangulation spaces can be given a discrete or continuous structure.  The continuous space is a foliated space as constructed by Moore and Schochet \cite{moore}, and the discrete space is a transversal of that space.  Although our work uses both of these structures, our results will focus primarily on the discrete space, since it is generally easier to work with and has the same large-scale properties.

Because we are modeling random triangulations, our results involve measures, which can be approximated by sequences of measures on spheres or on a single leaf.  In particular, approximating via spheres gives a way to approximate the IDS via computing eigenvalues on finite spaces.

Fabila Carrasco et al.\,\cite{fabila} studied a discrete magnetic Laplacian on graphs, and we are interested in studying similar operators in our setting:
\[ \Delta_{\mathrm{disc}} f(x_1, z_1) = \frac{1}{w(x_1)} \sum_{y_1 \in N(x_1)} \bar{w}(x_1, y_1) \left( f(y_1, z_1) - f(x_1, z_1) \right) \]
and
\[ \Delta_{\mathrm{disc}, V, B} f(x, z) = \frac{1}{w(x_1)} \sum_{y_1 \in N(x_1)} \bar{w}(x_1, y_1) \left( e^{i \alpha(x_1, y_1)} f(y_1, z_1) \right) + (V(x_1) - 1) f(x_1, z_1), \]
where $w$ and $\bar{w}$ are weight functions on edges and vertices, respectively, and $\alpha$ and $V$ are data related to the magnetic field and potential.  Although our results are formulated with these operators in mind, they apply to a large class of operators.

Our major results are as follows.  In these results, $\hat{G}$ is a groupoid consisting of twice-marked triangulations, $\nu$ is a transverse measure on $\hat{G}$, and $W^*(\hat{G})$ is a von Neumann algebra that we will define on $\hat{G}$ using the measure $\nu$.

\textbf{Theorem \ref{big1}}. \emph{Suppose that $\nu$ is such that, for a.e. $x$, $G^x$ is recurrent and $\nu^x = \nu$.  If $\nu(x) > 0$ for some (equivalently, a.e.) $x$, then $W^*(\hat{G})$ is a hyperfinite type $I_{|\hat{G}_0|}$ factor.  Otherwise, $W^*(\hat{G})$ is a hyperfinite type $II_1$ factor.}

\textbf{Theorem \ref{big2}}. \emph{If $H$ has finite hopping range, and $\nu'_n$ is a sequence of measures converging to $\nu$, then we can obtain the density of states $\kappa_H(\phi)$ as a limit of the leafwise local trace} $\mathrm{Tr} (\phi(H))$ \emph{averaged with respect to $\nu'_n$; that is,}
\[ \kappa_H(\phi) = \lim_{n \to \infty} \mathrm{Tr} \left(\nu'_n(\hat{G}^A) \phi(H) \right). \]

\textbf{Theorem \ref{big3}}. \emph{Suppose that $\mathcal{D}$ is discrete, $\hat{G}$ is ergodic, and that $H \in W^*(\hat{G})$ is $G$-invariant and has finite hopping range.  Suppose also that $t \in \mathbb{R}$.  The following are equivalent:}

\begin{enumerate}
	\item \emph{The density of states $\kappa_H(t) > 0$.}
	\item \emph{For some $x$, $H|_{\hat{G}^x}$ has an eigenfunction with eigenvalue $t$ supported on some finite patch $A$ with $\nu(G^A) > 0$. }
	\item \emph{For almost all $x$, $\ker (H_{\hat{G}^x} - \lambda I)$ is nontrivial and spanned by compactly supported eigenfunctions.}
\end{enumerate}

\section{Notation}

\begin{description}
\item[$A_r$] the $r$th interior of the discrete decorated finite patch $A$
\item[$C^*_r(G), C^*_r(\hat{G}$)] the reduced $C^*$-algebra associated with the groupoid $G$ or $\hat{G}$
\item[$d_D$] the distance on the decoration space
\item[$D$] the diagonal function on $\hat{G}$
\item[$\mathcal{D}$] the decoration space
\item[$g_{A, T}$] the inclusion map of a decorated (discrete or continuous) finite patch $A$ into the triangulation $T$
\item[$G$] the holonomy groupoid: the set of triangulations with two marked tangent vectors, up to automorphisms
\item[$G_0$] the set of units in the holonomy groupoid: the set of triangulations with a marked tangent vector, up to automorphisms
\item[$\hat{G}, \hat{G}_0$] the discrete holonomy groupoid, and its set of units
\item[$G_0^A, G^A, \hat{G}_0^A, \hat{G}^A$] the set of (discrete, continuous) triangulations containing the (marked, twice-marked) finite patch $A$
\item[$N(x)$] the set of neighbors of a vertex $x$
\item[$S(T)$] the discrete unit tangent bundle of $T$
\item[$S(|T|)$] the unit tangent bundle of $|T|$
\item[$|T|$] the smooth geometric realization of $T$
\item[$\mathcal{T}$] the set of abstract simplicial complexes
\item[$W^*(G), W^*(\hat{G}$)] the von Neumann algebra associated with the groupoid $G$ or $\hat{G}$
\item[$x_A$] the marked tangent vector on the (discrete or continuous) finite patch $A$
\item[$x_1, x_2$] the tail and head of the discrete tangent vector $x$
\item[$\delta$] the metric defined on $G_0$
\item[$\kappa_H$] the density of states associated with the operator $H \in W^*(\hat{G})$ 
\item[$\nu$] the transverse measure on $G_0$
\item[$\tau$] the trace on $W^*(G)$ or $W^*(\hat{G})$
\item[$\psi_{A, B}$] the coordinate patch associated with the patch $A$ and ball $B$
\item[$\Omega_A$] the set of permissible decorations on the decorated finite patch $A$

\end{description}

\section{Triangulation spaces}

\subsection{The space of pointed triangulations}

\begin{defi}[Triangulation; geometric realization]
	By a triangulation we mean an abstract simplicial complex $T'$ of dimension 2.  The sets of vertices, edges, directed edges, and faces of $T'$ will be denoted $V(T')$, $E(T')$ $\overrightarrow{E}(T')$, and $F(T')$.  The geometric realization of $|T'|$ of $T'$ is a metric space formed by assigning to each vertex $v$, edge $e$, and face $f$ a Euclidean simplex $|v|, |e|, |f|$ of the same dimension with marked points corresponding to vertices, respecting inclusion so that $|v|$ is identified with the corresponding point in $|e|$ if $v$ is a vertex of $e$, and $|e|$ is identified with the corresponding segment in $|f|$ if $e$ is an edge of $f$.  The usual metric on $|T'|$ is simply the Euclidean metric on each simplex, joined by shortest paths.
\end{defi}

Fix an integer $d > 6$ and a compact metric space $\mathcal{D}$.  Let $d_D$ denote the metric on $\mathcal{D}$.

\begin{defi}[Decorated triangulation]
	Let $\mathcal{T}$ be the set of all ordered pairs $T = (T', D_T)$, where $T'$ is a triangulation satisfying:
	\begin{itemize}
	\item each vertex of $T'$ has degree at most $d$, and
	\item the geometric realization $|T'|$ of $T'$ is a surface of genus 0,
	\end{itemize}
	and where $D_T: \overrightarrow{E}(T') \to \mathcal{D}$.  That is, a decorated triangulation $T$ consists of a triangulation $T'$ with decorations $D$.  (These decorations are analogous to what in the study of tilings are called markings; we call them decorations because we have a different use in mind for the term ``markings''.)  Any simplicial properties of $T'$ will be regarded as properties of $T$; e.g., we will define $V(T) = V(T')$.  We require isomorphisms between triangulations to preserve decorations: $T$ is isomorphic to $U$ if and only if there is a simplicial isomorphism $\phi: T' \to U'$ such that $D_{U}(\phi(e)) = D_{T}(e)$ for all $e \in \overrightarrow{E}(T_1)$.  Although we have defined decorations on edges, we could also speak of decorations on vertices, for example by considering triangulations where some components of $D$ are required to be equal for all edges emanating from a vertex.
	\end{defi}

Next we will define a geometric structure on such a triangulation.  There are many possible ways to define such a structure, including some that take decorations into account.  For our purposes, the following structure will suffice.

\begin{defi}[The smooth geometric realization of a triangulation]
	Let $T \in \mathcal{T}$, and let $|T|$ denote its geometric realization.  Denote by $d_v$ the degree of $v$ for each vertex $v$ of $T$.  For each directed edge $(v, w)$ of $T$, map $B(0, 2/3) \subset \mathbb{R}^2$ to $B(v, 2/3) \subset |T|$ by preserving the distance from 0 or $v$, and subdividing the unit circle into $d_v$ equal intervals and mapping the $k$th interval linearly to the angles on the $k$th face counterclockwise from $w$.  These maps form charts of $|T|$.  Let $\langle \cdot, \cdot \rangle_v$ be the Euclidean inner product on $B(0, 2/3)$ in the aforementioned chart, and let $\langle \cdot, \cdot \rangle_{pl}$ be the Euclidean inner product on the usual piecewise linear structure of $|T|$.  We observe that $\langle \cdot, \cdot \rangle_v$ does not depend on the choice of the second vertex $w$, and that $\langle \cdot, \cdot \rangle_{pl}$ is defined on all points except for vertices.  Hence we can define a Riemannian metric
	\[ \langle \cdot, \cdot \rangle = j(r) \langle \cdot, \cdot \rangle_v + (1 - j(r)) \langle \cdot, \cdot \rangle_{pl}, \]
	where $j$ is a smooth function with $j(0) = 1$ and $j(r) = 0$ for $r \geq \frac{1}{3}$, and $r$ is the distance from the point $x$ to the nearest vertex.
\end{defi}

Let $S(|T|)$ denote the unit tangent bundle of $|T|$, with the local product metric given by its structure as an $S^1$-fiber bundle.

We note that the isometry class of a face $|F|$ depends only on the degrees of the vertices of $F$.

For a triangulation $T \in \mathcal{T}$, we will denote by $\mathrm{Aut}(T)$ the group of simplicial automorphisms of $T$ (which is trivial in most cases).  The group $\mathrm{Aut}(T)$ acts naturally on $T$, $|T|$, and $S(|T|)$.

We are also interested in the discrete sphere bundle of a triangulation.

\begin{defi}[Discrete tangent vector; the discrete sphere bundle of a triangulation]
	Let $T$ be a triangulation.  A discrete tangent vector of $T$ is an ordered edge $(x_1, x_2)$ in $T$.  The discrete sphere bundle of $T$, denoted $S(T)$, is the set of discrete tangent vectors of $T$ equipped with the metric $d((x_1, x_2), (y_1, y_2)) = max(d(x_1, x_2), d(y_1, y_2))$.  Thus two distinct discrete tangent vectors $(x_1, x_2)$ and $(y_1, y_2)$ are adjacent if and only if $x_1$ and $y_1$ are either equal or adjacent, and $x_2$ and $y_2$ are either equal or adjacent.  We will often abbreviate, for example, $(x_1, x_2)$ to $x$.
\end{defi}

We can embed $S(T)$ in $S(|T|)$ by mapping to each $(x_1, x_2)$ the tangent vector at $|x_1|$ along the edge $|(x_1, x_2)|$.

Let $G_0$ be the set of pairs $(T, [x])$ where $[x]$ is an orbit of the action of $\mathrm{Aut}(T)$ on $S(|T|)$.  Since automorphisms are rare, we will usually think of such an orbit as a single point and write $(T, x)$ or simply $x$ if the context is clear.  When we define the holonomy groupoid later on, $G_0$ will be the set of units in that groupoid.

\subsection{Topology on the space of pointed triangulations}

Our next task is to define a topology on $G_0$.  Loosely speaking, this topology is defined by a metric in which points that are close together on the same triangulation to be close, and points on different triangulations are close if those triangulations agree on a large radius up to a small change in decorations.

\begin{defi}[Nearly decoration-preserving isometry]
	Let $T$ and $U$ be triangulations with $A \subset S(|T|)$.  If $\phi: A \to S(|U|)$ is an isometry such that, for every discrete tangent vector $x \in A$, $\phi(x)$ is a discrete tangent vector in $U$ with $d_D(D_T(x), D_U(\phi(x))) < \epsilon$, we say that $\phi$ is an $\epsilon$-nearly decoration-preserving isometry (henceforth $\epsilon$-NDPI).
\end{defi}

\begin{defi}[Asymmetric distance in $G_0$]
	Suppose that $(T, [x]), (U, [y]) \in \mathcal{T}$.  For each $x' \in [x]$, $r_\phi \geq 0$, and $\phi: \overline{B}_{r_\phi}(x') \to S(|U|)$ an $\epsilon_\phi$-NDPI, let $d_\phi$ be the distance from $\phi(x')$ to $[y]$.  We define 
	\[ \hat{\delta}((T, [x]), (U, [y])) = \inf \max(e^{-r_\phi}, ed_\phi, \epsilon_\phi), \]
	where the infimum is over all $x'$, $r_\phi$, and $d_\phi$ satisfying these conditions.
	\end{defi}

\begin{lem}
	The function $\hat{\delta}$ defined this way is nondegenerate; that is,
	\[ \hat{\delta}((T, [x]), (U, [y])) = 0 \]
	only if $(T, [x]) \cong (U, [y])$.
	\end{lem}

\begin{proof}
If $\hat{\delta}((T, [x]), (U, [y])) = 0$, then there exist $x'$ and $\phi$ as in the definition with arbitrarily large $r_\phi$ and arbitrarily small $d_\phi$ and $\epsilon_{\phi}$.  By composing with automorphisms if necessary, we may assume without loss of generality that $x' = x$ and that $d(\phi(x), y)$ is arbitrarily small.  Let $\phi_n$ be such that $r_{\phi_n} > n$ and $d(\phi_n(x), y) < 1/n$.  Each $z \in S(|T|)$ can be written as $\exp_x(v)$ for some $v \in \mathbb{R}^3$.  For sufficiently large $n$, we have $\phi_n(B_{|v|}(x)) \in B_n(y)$. Then $\phi_n(z) = \exp_{\phi_n(x)}(v)$ and $\phi_n(z) \to \exp_y(v)$.  This means that
\[ \phi := \lim_{n \to \infty} \phi_n, \]
where the limit is taken pointwise, is a well-defined isometry from $S(|T|)$ to $S(|U|)$ that sends $x$ to $y$.  Furthermore, if $z$ is a discrete tangent vector in $T$, then $d_D(D_T(z), D_U(\phi_n(z))) < \epsilon$ for every $\epsilon > 0$ because some $\phi_n$ is $\epsilon$-nearly decoration-preserving, and hence $D_T(z) = D_U(\phi(z))$.  
\end{proof}

\begin{lem}
	The function $\hat{\delta}$ defined this way satisfies the triangle inequality:
	\[ \hat{\delta}((T, [x]), (V, [z])) \leq \hat{\delta}((T, [x]), (U, [y])) + \hat{\delta}((U, [y]), (V, [z])). \]
\end{lem}

\begin{proof}
We can always take $r_\phi$ to be $0$ and $\phi: \{x\} \to \{y\}$ to be the map sending $x$ to $y$.  Hence
\[ \hat{\delta}((T, [x]), (U, [y])) \leq e, \]
and we can restrict our attention to maps with $d_\phi < 1$.

If $(T, [x]), (U, [y]), (V, [z]) \in \mathcal{T}$, and we have a pair of maps $\phi_1: \overline{B}_{r_{\phi_1}}(x') \to S(|U|)$ and $\phi_2: \overline{B}_{r_{\phi_2}}(y') \to S(|V|)$, we can compose them as follows.  By composing with an action of $\mathrm{Aut}(U)$, we can assume without loss of generality that the distance from $\phi_1(x')$ to $y'$ is at most $d_{\phi_1}$.  We can then define
\[ r_\phi = \min(r_{\phi_1}, r_{\phi_2} - d_{\phi_1}), \]
\[ \epsilon_\phi = \epsilon_{\phi_1} + \epsilon_{\phi_2}, \]
and
\[ \phi =\phi_2 \circ \phi_1. \]
If $w \in \overline{B}_{r_{\phi_1}}(x')$  is a discrete tangent vector, we have
\begin{eqnarray*}
d_D(D_T(w), D_V(\phi(w))) & \leq & d_D(D_T(w), D_U(\phi_1(w))) + d_D(D_U(\phi_1(w)), D_V(\phi(w))) \\
& \leq & \epsilon_{\phi_1} + \epsilon_{\phi_2} \\
& = & \epsilon_\phi,
\end{eqnarray*}
so $\phi$ is $\epsilon_\phi$-nearly decoration-preserving. Then
\begin{eqnarray*}
ed_\phi & \leq & ed_{\phi_1} + ed_{\phi_2} \\
& \leq & \left( \hat{\delta}((T, [x]), (U, [y])) + \hat{\delta}((U, [y]), (V, [z])) \right)
\end{eqnarray*}
and
\begin{eqnarray*}
e^{-r_\phi} & \leq & \max(e^{-r_{\phi_1}}, e^{d_{\phi_1} - r_{\phi_2}}) \\
& \leq & \max(e^{-r_{\phi_1}}, ed_{\phi_1} + e^{-r_{\phi_2}}) \\
& \leq & \hat{\delta}((T, [x]), (U, [y])) + \hat{\delta}((U, [y]), (V, [z])),
\end{eqnarray*}
where the second inequality comes from the fact that $d_{\phi_1} < 1$ and
\[ \frac{d}{dt} e^t < e \]
for $t < 1$.
\end{proof}

\begin{defi}[Distance in $G_0$]
	We define
	\[ \delta((T, [x]), (U, [y])) = \hat{\delta}((T, [x]), (U, [y])) + \hat{\delta}((U, [y]), (T, [x])). \]
\end{defi}

\begin{thm}
	The function $\delta$ defined in this way is a metric.
\end{thm}

\begin{proof}
	Since $\hat{\delta}$ is nondegenerate and satisfies the triangle inequality, so does $\delta$.  Furthermore, $\delta$ is symmetric by definition. 
\end{proof}

\begin{defi}[The triangulation topology]
	The triangulation topology is the topology induced by the metric $\delta$.
\end{defi}

\subsection{Foliated structure}
Next, we will give $G_0$ a foliated structure as defined by, for example, Moore and Schochet \cite{moore}.  Foliated spaces differ from classical foliations in that the total space is not required to be a manifold; instead, a foliated space locally looks like the product of $\mathbb{R}^n$ with some model space (often, as in this case, a Cantor-like set).  In this case our model space will be the discrete space $\hat{G}_0$ with the metric defined above.  Short distances in this metric correspond to small motions of the marked tangent vector (motion in the leafwise direction), changes in the triangulation far from the marked tangent vector (motion in the transverse direction), and small changes in decorations on the triangulation (also motion in the transverse direction).

We will start by defining a useful family of subsets of $G_0$ and $\hat{G}_0$.  These will be used to define the foliated structure, but, more importantly, will be a basis of the $\sigma$-algebra upon which our transverse measure is defined.

\begin{defi}[Decorated finite patch]
	By a decorated discrete finite patch we mean an ordered triple $(A, x, \Omega)$ where $A$ is a finite triangulation, $x$ is a marked discrete tangent vector of $A$, and $\Omega$ is a Borel subset of $\prod_{E(A)} \mathcal{D}$ (identified with the space of functions from $E(A)$ to $\mathcal{D}$).  Generally we will refer to such a patch simply as $A$ and write $x_A$ and $\Omega_A$.  We similarly define a decorated continuous finite patch as an ordered triple where $A$ and $\Omega$ are as above and $x \in S(|A|)$.  We call $\Omega$ the set of permissible decorations on $A$. If $|\Omega| = 1$ we say the finite patch is exact.  Likewise, we define a twice-marked decorated (discrete or continuous) finite patch as an ordered quadruple $(A, x, y, \Omega)$ where $x$ and $y$ are marked (discrete or continuous) tangent vectors.
	\end{defi}

\begin{defi}[Sets of triangulations with a particular decorated finite patch]
	Let $A$ be a decorated discrete finite patch.  We define $\hat{G}_0^A \subset \hat{G}_0$ to be the set of triangulations with a marked discrete tangent vector $(T, [x])$ such that there exists $g_{A, T}: A \to T$ with the properties that:
	\begin{itemize}
		\item $g_{A, T}$ is a simplicial embedding;
		\item $g_{A, T}(x_A) = x$;
		\item the function $x \mapsto D_T(g_{A, T}(x)) \in \Omega_A$.
	\end{itemize}
	For any particular $A$ and $T$, if such a $g_{A, T}$ exists, it is unique up to isomorphism (or choice of $x \in [x]$).  We analogously define $\hat{G}^A \subset \hat{G}$ if $A$ is a twice-marked decorated discrete finite patch, $G_0^A \subset G_0$ if $A$ is a decorated continuous finite patch, or $G^A \subset G$ if $A$ is a twice-marked decorated continuous finite patch.
\end{defi}

\begin{defi}[Coordinate patches of $G_0$]
	Let $A$ be a decorated continuous finite patch.  Let $B: S \to A$ be any smooth, but not necessarily isometric, embedding of a region $S \subset \mathbb{R}^3$ into $A$.  We define $\psi_{A, B}: \hat{G}_0^A \times B$ to $G_0$ by
	\[ \psi_{A, B}(T, s) = (T, [g_{A, T}(B(s))]), \]
	and call it the coordinate patch associated with the patch $A$ and region $B$.
\end{defi}

We can think of such a patch as a triangulation that is fixed in the region $A$ and allowed to vary outside of $A$ and whose decorations are allowed to vary within $\Omega_A$, with a marked tangent vector that is allowed to move within the subset $B$.

\begin{lem}
	The space $G_0$ equipped with the coordinate patches $\psi_A$ is a foliated space.
\end{lem}
\begin{proof}
Since $B$ in the definition contains at most one point in any $\mathrm{Aut}(T)$ orbit of $S(|T|)$ for any $T$ containing $A$, the coordinate patch $\psi_A$ is one-to-one.  It follows from the manifold structure of $S(|T|)$ that $\psi_{A', B'}^{-1}\psi_{A, B}$ is smooth.  
\end{proof}

\subsection{Holonomy groupoid}

Our next objective is to construct $G$, the holonomy groupoid of $G_0$.  The precise construction of the holonomy groupoid is given in Moore and Schochet \cite{moore}.

Briefly, a groupoid may be thought of as a small category in which all morphisms are invertible.  It consists of a groupoid $G$ with a space of units $G_0$, a diagonal map $\Delta: G_0 \to G$, an inversion map $\cdot^{-1}$ on $G$, range and source maps $r, s: G \to G_0$, and an associative multiplication on pairs $(u, v)$ where $r(v) = s(u)$.  These must satisfy $\Delta r(x) = \Delta s(x) = x$, $u \cdot \Delta s(u) = \Delta r(u) \cdot u = u$, $r(u^{-1}) = s(u)$, and $u u^{-1} = \Delta r(u)$.  We denote $G_x = s^{-1}(x)$ and $G^x = r^{-1}(x)$.

One simple example of a groupoid is an equivalence relation $\sim$, where $(x, y) \in G$ iff $x \sim y$, and we have $r(x, y) = x$, $s(x, y) = y$, $(x, y)^{-1} = (y, x)$, and $(x, y)(y, z) = (x, z)$.  In fact, we will see that our holonomy groupoid is such a groupoid in the absence of symmetry.

In general the holonomy groupoid is constructed as follows.  By a plaque we mean a connected component of $\ell \cap U$ where $\ell$ is a leaf and $U$ is an open set.  Given a path $\gamma$ from $x$ to $y$, and  transversals $N_x$ and $N_y$ through $x$ and $y$, respectively, we cover $\gamma$ by small open sets $V_1, \ldots, V_n$.  If the $V_i$ are sufficiently small, any plaque in $V_i$ intersects a unique plaque in $V_{i + 1}$, so that we can start with the plaque in $V_1$ containing some $x' \in N_x$ and follow the sequence of intersecting plaques to a unique plaque in $V_n$, which intersects a unique point $y'$ in $N_y$.  Hence this results in a map from a neighborhood in $N_x$ to one in $N_y$, realized by following a path from $N_x$ to $N_y$ sufficiently close to $\gamma$.  The germ of this map is independent of the choice of $V_i$, and two paths correspond to the same holonomy element if the germs of their corresponding maps on $N_x$ and $N_y$ are the same.

\begin{lem}
	The holonomy groupoid $G$ consists of $\mathrm{Aut}(T)$ orbits of pairs of points $[(x, y)]$, where $x, y \in S(|T|)$ for some $T \in \mathcal{T}$.
\end{lem}

\begin{proof}
The leaves of $G_0$ correspond to $S(|T|) / \mathrm{Aut}(T)$ with $T$ a triangulation.  A plaque on a leaf of $G_0$ corresponds to a small open subset of $S(|T|)$.

Suppose that $x, y \in S(|T|)$ and $\gamma$ is a path from $x$ to $y$.  We can choose a sufficiently large decorated continuous finite patch $A$ such that $x = g_{A, T}(x_A)$ and $\gamma$ is contained in the image of $g_{A, T}$.  Suppose additionally that $V_1, \ldots, V_n$ are open subsets of $S(|A|)$ such that the $g_A(V_i)$ cover $\gamma$.  

Suppose that $(T', x')$ is another triangulation with $x' = g_{A, T'}(x_A)$.   Let $y' = g_{A, T'}(g_{A, T}^{-1}(y))$; define $y'$, $\gamma'$, and $V'_i$ analogously.

Then $V'_1, \ldots, V'_n$ cover $\gamma'$, which is a path from $x'$ to $y'$.  It follows that the holonomy element corresponding to $\gamma$ sends $[x']$ to $[y']$; it depends on $x$ and $y$ but not on $\gamma$ itself.  This holds for any $(T', x') \in G_0^A$.  Hence there is a unique holonomy element for each $\mathrm{Aut}(T)$ orbits of pairs of points $[(x, y)]$ on some triangulation $T$.
\end{proof}

\begin{defi}[The discrete holonomy groupoid and its space of units]
	The holonomy groupoid $G$ contains a subgroupoid $\hat{G}$ consisting of elements of the form $[(|x|, |y|)]$ where $x$ and $y$ are discrete tangent vectors in $S(T)$ for some triangulation $T$.  We denote its set of units by $\hat{G}_0$.
\end{defi}

If $\mathrm{Aut}(T)$ is trivial, as we expect to be true in many cases, we can simply think of an element of $G$ as a triangulation with two marked tangent vectors, and an element of $\hat{G}$ as a triangulation with two marked discrete tangent vectors.  In any case, we will often suppress the brackets and write such an element $(x, y)$.

Our primary objects of interest will be $\hat{G}$ and $\hat{G}_0$.  However, we will often need to refer to $G$ and $G_0$ in order to use results about foliated spaces, as $\hat{G}_0$ does not have (for our purposes) a particularly useful foliated structure in its own right (for example, its leaves would consist of a single point).

\subsection{Transverse measure}\label{transverse}

The notion of measure needed to consider a "random triangulation" is a transverse measure.  Moore and Schochet \cite{moore} define a transversal of a groupoid $G$ as a Borel subset of its space of units $G_0$ whose intersection with each equivalence class (leaf, in this case) is countable, and a transverse measure as a measure on transversals satisfying certain properties.  We will not repeat the precise definition here, as we intend to use a simpler characterization, also from Moore and Schochet \cite{moore}: a transverse measure is equivalent to a measure on a single complete transversal (that is, a transversal that intersects every leaf).  Since $\hat{G}_0$ is such a transversal, we will simply think of our transverse measures as measures on $\hat{G}_0$.

We will construct our measures as limits of finitely supported ones, taken in the following sense:

\begin{defi}[Convergence of transverse measures]
  We say that a sequence of transverse measures $\nu_1, \nu_2, \ldots$ converges to a transverse measure $\nu$ if $\nu_k(\hat{G}_0^A) \to \nu(\hat{G}_0^A)$ for every decorated discrete finite patch $A$.
\end{defi}

There are (at least) two useful means of obtaining a measure as a limit of finitely supported measures.

\begin{defi}[Transverse measures constructed as limits of measures on spheres]\label{spheremeasure}
Let $\hat{G}_0^{S^2}$ be the subset of $\hat{G}_0$ consisting of points on finite (sphere) triangulations, and $\nu_k$ be a sequence of measures on $\hat{G}_0^{S^2}$.  We require that $\nu_k(x) = \nu_k(y)$ for $(T, x, y) \in \hat{G}$.  In other words, the restriction of each $\nu_k$ to a single sphere is a discrete uniform measure.
\end{defi}

There are many possible choices for $\nu_k$.  A few examples may be instructive:
\begin{enumerate}
\item Consider the case where $\mathcal{D}$ is trivial and $\nu_k$ is a uniform probability measure on the sphere triangulations with at most $k$ faces.  To our knowledge, it has not been proven whether or not the resulting sequence of measures converges, but in any event it is easy to show via diagonalization that a subsequence of the $\nu_k$ converges.  That is, let $A_1, A_2, \ldots$ be an enumeration of decorated finite patches that includes, for every finite patch $A$, a basis of the Borel $\sigma$-algebra on  $\prod_{E(A)} \mathcal{D}$.  Let $\nu_{0, j} = \nu_j$.  For each $k$, choose a subsequence $\nu_{k, j}$ of $\nu_{(k - 1), j}$ such that $\nu_{k, j} \to p_k$ for some $p_k \in [0, 1]$ (by compactness, such a subsequence must exist).  Then the diagonal sequence $\nu_{j, j}$ must have $\nu_{j, j}(A_k) \to p_k$ for each $k$.  Henceforth we will assume that such a sequence $\nu_k$ is fixed, and denote the limit by $\nu_u$.

\begin{con}\label{thisconjecture}
	The sequence of measures $\nu_k$ in this construction converges.
\end{con}

\item Let $\mathcal{D} = \{0, 1\}$ and let $\alpha_k$ be the sphere triangulation formed by identifying the boundaries of two balls of radius $k$ on a 2-dimensional triangular grid.  Define $\nu_k$ to be the measure that decorates each edge of $\alpha_k$ independently with 0 or 1 with equal probability and chooses a marked edge uniformly, and let $\nu_f$ be the limit of the $\nu_k$.  If $A$ is a finite patch that is not a subset of a 2-dimensional triangular grid, then $\nu_f(A) = 0$ since the identified boundaries on each $\alpha_k$ have measures approaching 0.  Hence $\nu_f$ is a measure on decorated triangular grids, and $\hat{G}$ is essentially a discrete tiling space.  Without the decorations almost every leaf would be periodic and $\hat{G}_0$ would consist of a single point.  However, with the decorations almost every leaf is aperiodic.

\item We can construct spheres via a process analogous to a substitution tiling.  Let $\mathcal{D}$ be a finite set.  Then there are only finitely many possible decorated faces.  Choose a map $\beta$ that assigns to each decorated face with an ordering on the vertices $(F, a, b, c)$ a finite patch with three distinct marked points $(\beta(F), \beta(a), \beta(b), \beta(c))$ such that $\beta(F) \cong D^2$ and $\beta(a)$, $\beta(b)$, and $\beta(c)$ are on the boundary of $\beta(F)$ with no interior edges.  Let $\beta(\overline{ab})$ denote the portion of the boundary from $\beta(a)$ to $\beta(b)$ that does not pass through $\beta(c)$, and define $\beta(\overline{bc})$ and $\beta(\overline{ca})$ analogously.  We require that $\beta(\overline{ab})$ depend only on $D(a)$ and $D(b)$ and the analogous conditions hold for $\beta(\overline{bc})$ and $\beta(\overline{ca})$.  We can define $\beta$ on any triangulation by applying it to each face.  We say that $\beta$ is nondegenerate if there exists $k$ such that $\beta^k(F)$ has at least one interior face for every face $F$.  This implies that there exists $r < 1$ such that the proportion of boundary faces to total faces in $\beta^n(F)$ is at most $r^n$.  Then we can set $T_0$ to be any sphere triangulation, set $T_n = \beta^n(T_0)$, and choose $\nu_k$ to be the uniform measure on $T_k$.  Since $\beta$ is nondegenerate, almost all faces are in the interior (and eventually arbitrarily far from the boundary) of $\beta^n(F)$ for some $F$, and we can determine the limiting frequency of each patch from analyzing $\beta$, as in tiling theory.  We denote the limit by $\nu_s$.

\item Let $\theta_1$ be the finite triangulation consisting of a single vertex of degree 7, its neighboring faces, and their vertices and edges.  For $n > 0$, let $\theta_n$ be constructed from $\theta_{n - 1}$ by adding a face on every external edge, then adding additional vertices adjacent to each of the previous external vertices so that it becomes a vertex of degree 7, and connecting the newly added vertices in a cycle.  The first three iterations are shown in \ref{hyper}.

\begin{figure}
\centering
\begin{tikzpicture}
\draw (0, 2) -- (-2, 1) -- (0,0) -- (0, 2) -- (2, 1) -- (0, 0) -- (2, -1) -- (2, 1);
\draw (1, -2) -- (2, -1) -- (0, 0) -- (1, -2) -- (-1, -2) -- (0, 0) -- (-2, -1) -- (-1, -2);
\draw (-2, 1) -- (-2, -1);
\draw [red] (-2, 3) -- (0, 2) -- (-1, 4) -- (-2, 3) -- (-2, 1) -- (-3, 2) -- (-2, 3);
\draw [red] (-1, 4) -- (1, 4) -- (0, 2) -- (2, 3) -- (1, 4);
\draw [red] (3, 2) -- (2, 3) -- (2, 1) -- (3, 2) -- (4, 1) -- (2, 1) -- (4, 0) -- (4, 1);
\draw [red] (4, -1) -- (4, 0) -- (2, -1) -- (4, -1) -- (4, -2) -- (2, -1) -- (3, -3) -- (4, -2);
\draw [red] (1, -4) -- (3, -3) -- (1, -2) -- (1, -4) -- (0, -4) -- (1, -2);
\draw [red] (-1, -4) -- (-3, -3) -- (-1, -2) -- (-1, -4) -- (-0, -4) -- (-1, -2);
\draw [red] (-4, -1) -- (-4, 0) -- (-2, -1) -- (-4, -1) -- (-4, -2) -- (-2, -1) -- (-3, -3) -- (-4, -2);
\draw [red] (-4, 1) -- (-4, 0) -- (-2, 1) -- (-4, 1) -- (-3, 2);
\draw [blue] (-1, 6) -- (0, 6) -- (-1, 4) -- (-1, 6) -- (-2, 6) -- (-1, 4) -- (-3, 5) -- (-2, 6);
\draw [blue] (-3.5, 4.5) -- (-3, 5) -- (-2, 3) -- (-3.5, 4.5) -- (-4, 4) -- (-2, 3);
\draw [blue] (-4.5, 3.5) -- (-4, 4) -- (-3, 2) -- (-4.5, 3.5) -- (-5, 3) -- (-3, 2) -- (-5.5, 2.5) -- (-5, 3);
\draw [blue] (-6, 2) -- (-5.5, 2.5) -- (-4, 1) -- (-6, 2) -- (-6, 1) -- (-4, 1) -- (-6, 0.5) -- (-6, 1);
\draw [blue] (-6, 0) -- (-6, 0.5) -- (-4, 0) -- (-6, 0) -- (-6, -0.5) -- (-4, 0);
\draw [blue] (-6, -1) -- (-6, -0.5) -- (-4, -1) -- (-6, -1) -- (-6, -1.5) -- (-4, -1) -- (-6, -2) -- (-6, -1.5);
\draw [blue] (-6, -2.5) -- (-6, -2) -- (-4, -2) -- (-6, -2.5) -- (-6, -3) -- (-4, -2) -- (-5.5, -3.5) -- (-6, -3);
\draw [blue] (-4.5, -4.5) -- (-5.5, -3.5) -- (-3, -3) -- (-4.5, -4.5) -- (-3, -5.5) -- (-3, -3);
\draw [blue] (-2, -6) -- (-3, -5.5) -- (-1, -4) -- (-2, -6) -- (-1, -6) -- (-1, -4) -- (-0.5, -6) -- (-1, -6);
\draw [blue] (0, -6) -- (-0.5, -6) -- (0, -4) -- (0, -6) -- (0.5, -6) -- (0, -4);
\draw [blue] (1, 6) -- (0, 6) -- (1, 4) -- (1, 6) -- (2, 6) -- (1, 4) -- (3, 5) -- (2, 6);
\draw [blue] (3.5, 4.5) -- (3, 5) -- (2, 3) -- (3.5, 4.5) -- (4, 4) -- (2, 3);
\draw [blue] (4.5, 3.5) -- (4, 4) -- (3, 2) -- (4.5, 3.5) -- (5, 3) -- (3, 2) -- (5.5, 2.5) -- (5, 3);
\draw [blue] (6, 2) -- (5.5, 2.5) -- (4, 1) -- (6, 2) -- (6, 1) -- (4, 1) -- (6, 0.5) -- (6, 1);
\draw [blue] (6, 0) -- (6, 0.5) -- (4, 0) -- (6, 0) -- (6, -0.5) -- (4, 0);
\draw [blue] (6, -1) -- (6, -0.5) -- (4, -1) -- (6, -1) -- (6, -1.5) -- (4, -1) -- (6, -2) -- (6, -1.5);
\draw [blue] (6, -2.5) -- (6, -2) -- (4, -2) -- (6, -2.5) -- (6, -3) -- (4, -2) -- (5.5, -3.5) -- (6, -3);
\draw [blue] (4.5, -4.5) -- (5.5, -3.5) -- (3, -3) -- (4.5, -4.5) -- (3, -5.5) -- (3, -3);
\draw [blue] (2, -6) -- (3, -5.5) -- (1, -4) -- (2, -6) -- (1, -6) -- (1, -4) -- (0.5, -6) -- (1, -6);
\end{tikzpicture}
\caption{First three stages in the construction of $\nu_h$: $\theta_1$ in black, $\theta_2 \setminus \theta_1$ in red, and $\theta_3 \setminus \theta_2$ in blue.} \label{hyper}
\end{figure}
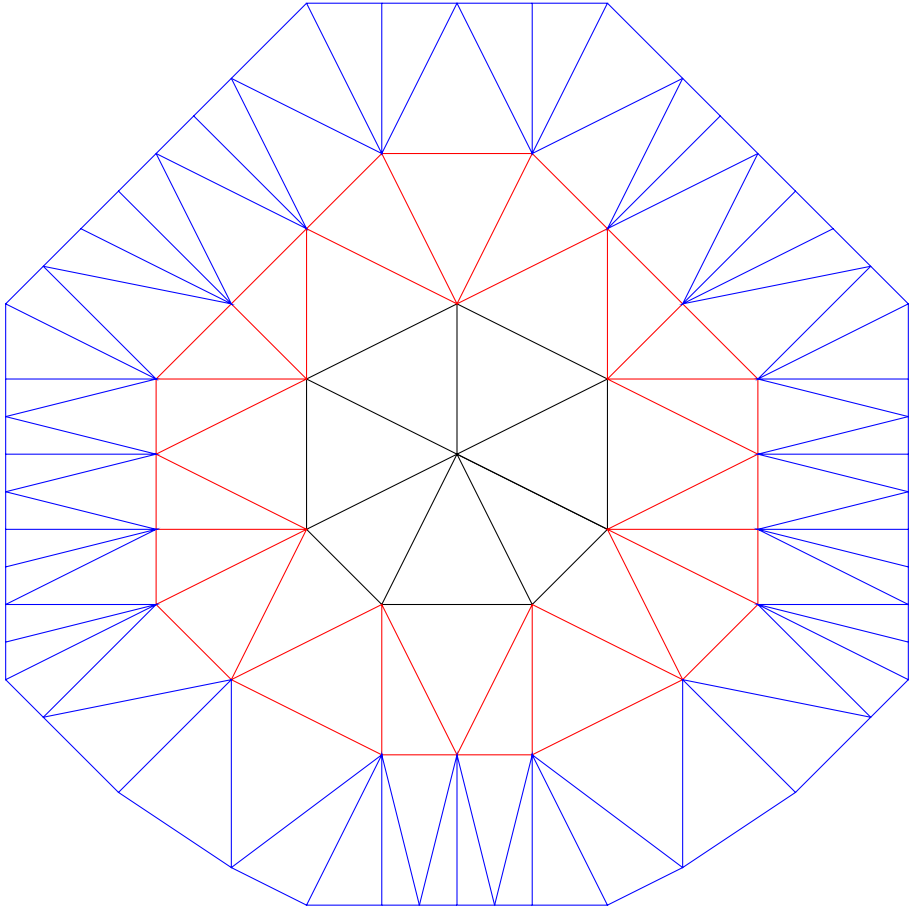

We can see that each boundary vertex has degree 3 or 4.  If $A$ denotes a vertex of degree 3 and $B$ denotes a vertex of degree 4, the boundary of $\theta_n$ is obtained from that of $\theta_{n - 1}$ via the substitution
\[ A \mapsto BAA; B \mapsto BA. \]
The largest eigenvalue of the matrix of this substitution is $\frac{3 + \sqrt{5}}{2}$, and hence the number of vertices in the boundary is asymptotically proportional to
\[ \left( \frac{3 + \sqrt{5}}{2} \right)^n; \]
in particular, it grows exponentially in $n$.  Furthermore, since $\frac{3 + \sqrt{5}}{2}$ is irrational, the limit of this substitution is aperiodic.

Now let $\Theta_k$ be the sphere triangulation formed by identifying the boundaries of two copies of $\theta_k$.  The boundaries become a cycle, which we will call the ridge, of vertices of degree 4 and 6 in a pattern given by the aforementioned substitution.  All vertices not on the ridge have degree 7, and the structure of vertices in a ball intersecting the ridge is determined by the ridge.  We define $\nu_h$ to be the limit of the measures $\nu_k$ that are uniform on the $\Theta_k$.  The proportion of vertices of degree 7 is asymptotically
\[ \frac{4}{5 + \sqrt{5}} < 1, \]
meaning that vertices of degree 4 and 6 have positive measure.  

It is not too hard to see that $\nu_h$ is well-defined.  Indeed, with some work, we could compute it for any ball of finite radius.  If the ball contains vertices of degree 4 and/or 6, and the vertices are consistent with the structure of $\Theta_n$ for large $n$, the ball's frequency is determined by the frequency of the pattern of ridge vertices (which can be determined with the usual techniques for substitution tilings).  All balls of radius $k$ with only degree 7 vertices are isomorphic, and their frequency is simply the limiting proportion of vertices that are a distance more than $k$ from the ridge.  Furthermore, because the limit of the substitution that produces the ridge is aperiodic, we need not consider automorphisms.

\end{enumerate}

Another means of constructing transverse measures uses a single leaf.  Not every leaf produces a well-defined transverse measure.  In general, it is more difficult to construct measures this way compared to constructing them on spheres.  However, this construction also leads to the idea of patch density, which will be useful in subsequent sections.

\begin{defi}[Transverse measures constructed as limits of measures on a leaf; patch density on a leaf]\label{leafmeasure}
	Let $x \in \hat{G}_0$ and let $f_1, f_2, \ldots$ be a sequence of functions on $\hat{G}^x$ such that $\int f_k \,d\lambda^x = 1$ for each $f_k$.  Often we will consider a leaf $\hat{G}^x$ that is recurrent in the sense of random walks and take $f_k$ to be the distribution of a random walk from $x$ after $k$ steps.  For a decorated discrete finite patch $A$, we define $\nu_k^x(A) = \int_{\hat{G}^{A} \cap \hat{G}^x} f_k \,d\lambda(x)$.

 	If $\nu^x(A) = \lim_{k \to \infty} \nu_k^x(A)$ exists, we call it the patch density of $A$ on $G^x$ (with respect to the exhaustion $(f_k)$).  If the exhaustion is not specified, we will assume it to be the exhaustion with respect to random walks.
\end{defi}

\section{Properties}

\subsection{Amenability}

Our goal is to study operators on $\hat{G}$ by averaging along leaves.  If $\hat{G}$ has two properties, amenability and ergodicity, such averages will be well-behaved.  Additionally, when we construct a von Neumann algebra on $\hat{G}$, these properties will be useful for classifying it.   We will first deal with amenability.

There are multiple equivalent criteria for amenability.  The most useful one to us is Reiter's criterion, which in our terminology is that there exists a sequence of functions $g_n: \hat{G} \to \mathbb{R}$ satisfying the following for $\nu$-a.e. $x$:
\begin{enumerate}
\item $\int g_n \,d\lambda_x = 1$, and
\item for all $y \in G^x$, we have $\lim_{n \to \infty} \int g_n(x, z) - g_n(y, z) \,d\lambda^x(z) = 0$.
\end{enumerate}

If almost all leaves had subexponential growth, we could simply take $g_n$ to be $\frac{1}{|B_n(x)|} \chi_{B_n(x)}(y)$.  This does not work in the case of, for example, $\nu_h$, and so we resort to an argument using random walks.

\begin{thm}
	If the groupoid $\hat{G}$ is given a transverse measure such that $G^x$ is recurrent for a.e. x, then $\hat{G}$ is amenable.
\end{thm}

\begin{proof}
For $T$ a triangulation, $x$ a vertex of $T$, and $n$ a positive integer, let $B^x_1, B^x_2, \ldots$ denote a simple random walk on $T$ (that is, on the 1-skeleton of $T$).  To make our notation less cumbersome, we will use the shorthand $x \xrightarrow[n]{} y$ to denote $B^x_n = y$.  Let $g_n$ be the distribution of $B^x$ after a uniformly random number of steps from $1, \ldots, n$; that is,
\[ g_n(x, y) = \frac{1}{n} \sum_{k = 1}^n P(x, y, k). \]

Let $\epsilon > 0$.

Let $A_x^y$ be the first time that the walk $B^x$ hits $y$ - that is, the smallest positive integer such that $B^x_{A_x} = y$.  If $T$ is recurrent and $t \in \mathbb{N}$ is sufficiently large, then $P(A_x^y > t) < \frac{\epsilon}{2}$ and $P(A_y^x > t) < \frac{\epsilon}{2}$.  We will use $x \xrightarrow{n} y$ to denote $A_x^y = n$.

Now suppose that $n > \frac{2t}{\epsilon}$.

We have
\begin{eqnarray*}
g_n(x, z) - g_n(y, z) & = & \frac{1}{n} \sum_{k = 1}^n \left( P(x \xrightarrow[k]{} z) - P(y \xrightarrow[k]{} z) \right) \\
& = & \frac{1}{n} \sum_{m = 1}^\infty P(x \xrightarrow{m} y) \sum_{k = 1}^n \left( P(x \xrightarrow[k]{} z | x \xrightarrow{m} y) - P(y \xrightarrow[k]{} z | x \xrightarrow{m} y) \right).
\end{eqnarray*}

Our next step is to decompose this sum into two parts: one with $m \leq t$ and the other with $m > t$.  (That is, either the random walk starting at $x$ hits $y$ before time $t$, in which case the distribution of a random walk starting at $x$ is similar to one starting at $y$, or it does not hit $y$ before time $t$, which is unlikely because of our choice of $t$.)

First, we will consider the sum where $m \leq t$; that is, the random walk starting at $x$ hits $y$ before time $t$.  If $m \leq t$, we have
\begin{eqnarray*}
	& & \sum_{k = 1}^n \left( P(x \xrightarrow[k]{} z | x \xrightarrow{m} y) - P(y \xrightarrow[k]{} z | x \xrightarrow{m} y) \right) \\
	& = & \sum_{k = 1}^m P(x \xrightarrow[k]{} z | x \xrightarrow{m} y) + \sum_{k = m + 1}^n P(x \xrightarrow[k]{} z | x \xrightarrow{m} y) \\
	& & - \sum_{j = 1}^{n - m} P(y \xrightarrow[j]{} z = z | x \xrightarrow{m} y) - \sum_{j = n - m + 1}^n P(y \xrightarrow[j]{} z | x \xrightarrow{m} y) \\
	& \leq & \sum_{k = m + 1}^n P(x \xrightarrow[k]{} z | x \xrightarrow{m} y) - \sum_{j = 1}^{n - m} P(y \xrightarrow[j]{} z | x \xrightarrow{m} y) + \sum_{j = n - m + 1}^n P(x \xrightarrow[j]{} z | x \xrightarrow{m} y) \\
	& = & \sum_{j = 1}^{n - m} \left( P(x \xrightarrow[j + m]{} z | x \xrightarrow{m} y) - P(y \xrightarrow[j]{} z = z | x \xrightarrow{m} y) \right) + \sum_{k = n - m + 1}^n P(x \xrightarrow[k]{} z | x \xrightarrow{m} y).
\end{eqnarray*}
Now, if $x \xrightarrow{m} y$, then $x \xrightarrow[m]{} y$ and $B_{j + m}^x$ is a random walk of $j$ steps starting at $y$.  Hence each term in the first sum is $0$, and we are left with
\begin{eqnarray*}
\sum_{k = 1}^n \left( P(x \xrightarrow[k]{} z | x \xrightarrow{m} y) - P(y \xrightarrow[k]{} z | x \xrightarrow{m} y) \right) & \leq & \sum_{k = n - m + 1}^n P(x \xrightarrow[k]{} z | x \xrightarrow{m} y) \\
& \leq & m \\
& \leq & t \\
\sum_{m = 1}^t P(x \xrightarrow{m} y) \sum_{k = 1}^n \left( P(x \xrightarrow[k]{} z | x \xrightarrow{m} y) - P(y \xrightarrow[k]{} z | x \xrightarrow{m} y) \right) & \leq & t \sum_{m = 1}^t P(x \xrightarrow{m} y) \\
& \leq & t.
\end{eqnarray*}

Next, consider the sum with $m > t$.  We have
\begin{eqnarray*}
& & \sum_{m = t + 1}^\infty P(x \xrightarrow{m} y) \sum_{k = 1}^n \left( P(x \xrightarrow[k]{} z | x \xrightarrow{m} y) - P(y \xrightarrow[k]{} z | x \xrightarrow{m} y) \right) \\
& \leq & \sum_{m = t + 1}^\infty P(x \xrightarrow{m} y) \sum_{k = 1}^n P(x \xrightarrow[k]{} z | x \xrightarrow{m} y) \\
& \leq & \sum_{m = t + 1}^\infty n P(x \xrightarrow{m} y) \\
& = & n P(A_x^y > t) \\
& < & \frac{n \epsilon}{2}.
\end{eqnarray*}

Hence
\begin{eqnarray*}
\sum_{m = 1}^\infty P(x \xrightarrow{m} y) \sum_{k = 1}^n \left( P(x \xrightarrow[k]{} z | x \xrightarrow{m} y) - P(y \xrightarrow[k]{} z | x \xrightarrow{m} y) \right) & \leq & t + \frac{n \epsilon}{2} \\
g_n(x, z) - g_n(y, z) & \leq & \frac{1}{n} \left( t + \frac{n \epsilon}{2} \right) \\
& = & \frac{t}{n} + \frac{\epsilon}{2} \\
& \leq & \frac{\epsilon}{2} + \frac{\epsilon}{2} \\
& = & \epsilon.
\end{eqnarray*}
Exchanging $x$ and $y$ gives
\[ g_n(y, z) - g_n(x, z) < \epsilon \]
and hence
\[ |g_n(x, z) - g_n(y, z)| < \epsilon, \]
and we have shown that $\hat{G}$ is amenable.
\end{proof}

\begin{cor}
	The groupoid $\hat{G}$ with any transverse measure $\nu$ constructed as in Definition \ref{spheremeasure} is amenable.
\end{cor}

\begin{proof}
Using the result of Benjamini and Schramm \cite{benjamini}, $T$ is recurrent for $\nu$-a.e. $(T, [x])$.  By the preceding lemma, it suffices to show that this implies $G^x$ is recurrent.  Recall that each point $y \in G^x$ is actually a discrete tangent vector, or directed edge, $(y_1, y_2) \in E(T)$, and that two such distinct tangent vectors $(y_1, y_2)$ and $(z_1, z_2)$ are adjacent if and only if $y_1$ and $z_1$ are either equal or adjacent, and $y_2$ and $z_2$ are either equal or adjacent.  We will use the flow theorem, which states that a graph $\Gamma$ is transient if and only if it has a flow $\rho$ with a single source, no sinks, and a finite energy
\[ \mathcal{E}(\Gamma) = \sum_{(x, y) \in E(\Gamma)} \left( \rho(x, y) \right)^2 < \infty. \]

Suppose that $G^x$, where $x$ is a tangent vector on the triangulation $T$, is transient.  Then there exists a flow $\rho$ with a single source $w = (w_1, w_2) \in G^x$ and finite energy.  Define a flow on $T$ by
\[ \rho'(y, z) = \sum_{\stackrel{y_2 \in B_1(y)}{z_2 \in B_1(z)}} \rho((y, y_2), (z, z_2)). \]
We note that 
\[ \mathrm{div}(\rho')(y) = \sum_{y_2 \in B_1(y)} \mathrm{div}(\rho)(y, y_2), \]
and hence $\rho'$ has a single source at $w$ and no sinks.

Furthermore,
\begin{eqnarray*}
\left( \rho'(y, z) \right)^2 & = & \left( \sum_{\stackrel{y_2 \in B_1(y)}{z_2 \in B_1(z)}} \rho((y, y_2), (z, z_2)) \right)^2 \\
& \leq & \sum_{\stackrel{y_2 \in B_1(y)}{z_2 \in B_1(z)}} \rho((y, y_2), (z, z_2))^2 |B_1(y)| |B_1(z)| \\
& \leq & (d + 1)^2 \sum_{\stackrel{y_2 \in B_1(y)}{z_2 \in B_1(z)}} \rho((y, y_2), (z, z_2))^2.
\end{eqnarray*}
Taking sums gives
\[ \mathcal{E}(\rho') \leq (d + 1)^2 \mathcal{E}(\rho) < \infty, \]
from which we conclude that $T$ is transient.  By contrapositive, if $T$ is recurrent, then $G^x$ is recurrent.
\end{proof}

The case of measures constructed using a leaf as in Definition \ref{leafmeasure} is less clear.
\begin{con}\label{thatconjecture}
	If $G^x$ is recurrent and $\nu_x$ exists, then $G^y$ is recurrent for $\nu_x$-a.e. y.
\end{con}
If this conjecture holds, it will also imply that $\hat{G}$ is amenable when equipped with the transverse measure $\nu_x$ constructed subject to these conditions.

\subsection{Ergodicity}

Next we consider ergodicity.  We will start by proving the following lemma:

\begin{lem}\label{ergodiclemma}
Suppose that $\nu$ is chosen such that, for any decorated finite patch $A$, $\nu^x = \nu$ for $\nu$-almost every $x$ (that is, the patch density of any finite patch $A$ almost surely does not depend on $x$).  For all $f, g \in L^1(G^0)$, we have
\[ \lim_{n \to \infty} \int f(x) \int P(x \xrightarrow[n]{} y) g(y) \,d\lambda^x(y) \,d\mu(x) = \int f \,d\mu \int g \,d\mu. \]
\end{lem}
\begin{proof}
	For each finite patch $A$, let $\phi(A) = \nu(\hat{G}^A)$.  If $A$ and $B$ are decorated finite patches and $f = \chi_{\hat{G}^A}$ and $g = \chi_{\hat{G}^B}$, the claim holds since both sides equal $\phi(A) \phi(B)$.  The claim in general follows from bilinearity and the fact that functions of the form $\chi_{\hat{G}^A}$ span a dense subset of $L^1(\hat{G}^0)$.
\end{proof}

\begin{thm}
	Suppose that $\nu$ is chosen such that, for any decorated finite patch $A$, $\nu^x = \nu$ for $\nu$-almost every $x$.  Then $G$ equipped with the measure $\nu$ is ergodic.
\end{thm}

\begin{proof}
Suppose that $f$ is an invariant, real-valued function (that is, $g(y) = f(x)$ for any $y \in G^x$) with $\int f \,d\mu = 0$.  Setting $g = f$ in Lemma \ref{ergodiclemma},
\begin{eqnarray*}
0 & = & \left( \int f \,d\mu \right)^2 \\
& = & \lim_{n \to \infty} \int f(x) \int P(x \xrightarrow[n]{} y) f(x) \,d\lambda^x(y) \,d\mu(x) \\
& = & \lim_{n \to \infty} \int (f(x))^2 \,d\mu(x) \\
\end{eqnarray*}
and hence $f \equiv 0$ $\mu$-a.e.  Since every such $f$ is zero, $G$ is ergodic.
\end{proof}

\begin{con}\label{theotherconjecture}
	The measure $\nu_u$ is ergodic.
\end{con}

\section{Operator algebras}

\subsection{Construction}

We are interested in the properties of operators on $G$.  To study these operators, we will consider operator algebras.  Both the $C^*$-algebra and the von Neumann algebra are generated using leafwise convolution operators.

\begin{defi}[The reduced $C^*$-algebra of $G$]
	We construct $C^*_r(G)$ as a $C^*$-algebra operating on $\bigoplus_{x \in G_0} L^2(G_x)$.  To each continuous, compactly supported function $f$ on $G$, and for each $x \in G_0$, we associate the operator $\pi_x(f)$ on $L^2(G_x)$ given by
	\[ \pi_x(f)(\phi)(z) = \int_{y \in G_x} \phi(y) f(y, z) \,d\lambda_x(y) \]
	for each $\phi \in L^2(G_x)$.
	The closure of all $\pi_{C^*}(f) = \{\pi_x(f): x \in G_0\}$ under the norm
	\[ \|\pi_{C^*}(f)\| = \sup_{(T, [x]) \in G} \|\pi_x(f)\| \]
	forms the $C^*$-algebra $C^*_r(G)$ associated to $G$.
\end{defi}

\begin{defi}[The von Neumann algebras of $G$ and $\hat{G}$]
	We construct $W^*(G)$ as a von Neumann algebra operating on $L^2(G)$.  Using the transverse measure $\nu$ defined above, we can construct the spaces $L^p(G)$.  For $f \in L^1(G)$ and $\phi \in L^2(G)$, we define the operator
	\[ \pi(f)(\phi)([(x, z)]) = \int_{y} f(x, y) \phi(y, z) \,d\lambda_x(y). \]
	The weak closure of all such $\pi(f)$ is the von Neumann algebra $W^*(G)$ associated with $G$.  That is, $H \in B(L^2(G))$ belongs to $W^*(G)$ if there exists a sequence $(f_n)$ such that $\langle \pi(f_n)\phi, \eta \rangle \to \langle H \phi, \eta \rangle$ for all $\phi, \eta \in L^2(G)$.
\end{defi}
	
	We can repeat this construction using $\hat{G}$ in place of $G$.  Since $\hat{G}$ is a complete transversal of $G$, we have
	\[ W^*(G) \cong W^*(\hat{G}) \otimes \mathcal{B}(L^2([0, 1])) \]
	by Moore and Schochet \cite{moore} (Proposition 6.21).

The $C^*$ and von Neumann algebras are related as follows.

\begin{lem}
	If $f$ is a continuous, compactly supported function, then
	\[ |\pi(f)| \leq |\pi_{C^*}(f)|. \]
\end{lem}

\begin{proof}
	For $\phi \in L^2(G)$,
	\begin{eqnarray*}
	|\pi(f)\phi| & = & \left(\int \int |\pi(f) \phi(x, y)|^2 \,d\lambda_x(y) \,d\mu(x) \right)^{1/2} \\
	& = & \left(\int \int |\pi_x(f) \phi(x, y)|^2 \,d\lambda_x(y) \,d\mu(x) \right)^{1/2} \\
	& = & \left(\int |\pi_x(f) \phi|_{G^x}|^2 \,d\mu(x) \right)^{1/2} \\
	& \leq & \left(\int |\pi_x(f)|^2 |\phi_{G^x}|^2 \,d\mu(x) \right)^{1/2} \\
	& \leq & \left(\int |\pi_{C^*}(f)|^2 \phi_{G^x}|^2 \,d\mu(x) \right)^{1/2} \\
	& = & |\phi| |\pi_{C^*}(f)|.
	\end{eqnarray*}
\end{proof}

\begin{cor}
	There exists an algebra homomorphism $\Phi: C^*_r(G) \to W^*(G)$ such that $\Phi(\pi_C^*(f))  = \pi(f)$.
\end{cor}

\begin{proof}
Let $H \in C^*_r(G)$.  Choose a sequence of continuous, compactly supported functions $(f_n)$ with $\pi_{C^*}(f_n) \to H$.  Since $\pi_{C^*}(f_n)$ is Cauchy, $\pi(f_n)$ is also Cauchy by the preceding lemma, and hence converges.  Let $\Phi(H) = \lim \pi(f_n)$.  To show that this is well-defined, let $(g_n)$ be another such sequence of functions.  Then $\pi_{C^*}(f_n - g_n) \to 0$ and the preceding lemma gives $\lim \pi(f_n) = \lim \pi(g_n)$.  That $\Phi$ is an algebra homomorphism follows from the definition.
\end{proof}

For our results the most important algebra is $W^*(\hat{G})$.  We will first characterize the operators in $W^*(\hat{G})$, then discuss its classification.

\begin{defi}[The diagonal function]
	Define $D: \hat{G} \to \mathbb{C}$ by $D(T, x, y) = 1$ if $x = y$ and $D(T, x, y) = 0$ otherwise.
\end{defi}

It is easily seen that $D \in L^1(\hat{G}) \cap L^2(\hat{G})$ and that $D$ is the identity of leafwise convolution; that is, $\pi(D) = 1$.

\begin{lem}
	Let $f, g \in L^2(\hat{G})$ and $h \in L^1(\hat{G})$.  Then $\langle h * g, f \rangle = \langle h, f * g^* \rangle$.
\end{lem}

\begin{proof}
We have
	\begin{eqnarray*}
		\langle h * g, f \rangle & = & \int (h * g)(x, y) f(x, y) \,d\lambda^x(y) \,d\mu(x) \\
		& = & \int h(x, z) g(z, y) f(x, y) \,d\lambda^x(z) \,d\lambda^x(y) \,d\mu(x) \\
		& = & \langle h, f * g^* \rangle.
	\end{eqnarray*}
\end{proof}

\begin{thm}\label{linearoperators}
Let $H$ be a bounded linear operator on $L^2(G)$.  The following are equivalent:
	\begin{enumerate}
		\item $H \in W^*(\hat{G})$.
		\item $H(D) * g = H(g)$ for every $g \in L^2(\hat{G})$.
		\item $H = \pi(h)$ for some $h \in L^2(\hat{G})$.
	\end{enumerate}
\end{thm}

\begin{proof}
	If $H$ satisfies (2), let $h = H(D)$.  For every $g \in L^2(\hat{G})$, we have $H(g) = H(D) * g = h * g$; that is, $H = \pi(h)$.
	
	If $H$ satisfies (3), we can approximate $h$ as a norm, hence weak, limit of compactly supported functions $h_k \in L^2(\hat{G})$.  For every $g, f \in L^2(\hat{G})$,
	\begin{eqnarray*}
		\lim_{k \to \infty} \langle h_k * g, f \rangle & = & \lim_{k \to \infty} \langle h_k, f * g^* \rangle \\
		& = & \langle h, f * g^* \rangle;
	\end{eqnarray*}
	that is, $H$ is a weak limit of the $\pi(h_k)$.
	
	Finally, suppose that $H$ satisfies (1).  Then we can write $H$ as the weak limit of $H_k = \pi(h_k)$ with $h_k$ compactly supported.  Define $h = H(D)$.  Since
	\[ \langle H_k(D) * g, f \rangle = \langle H_k(D), f * g^* \rangle, \]
	weak convergence implies
	\begin{eqnarray*}
		\langle H(D) * g, f \rangle & = & \langle H(D), f * g^* \rangle \\
		& = & \lim_{k \to \infty} \langle H_k(D), f * g^* \rangle \\
		& = & \lim_{k \to \infty} \langle H_k(D) * g, f \rangle \\
		& = & \lim_{k \to \infty} \langle H_k(g), f \rangle \\
		& = & \langle H(g), f \rangle. 
	\end{eqnarray*}
	Because $f$ was arbitrary, this implies $H(D) * g = H(g)$.
\end{proof}

This result is convenient because it allows us to reduce questions about operators in $W^*(\hat{G})$ to questions about convolutions of functions on $L^2(\hat{G})$.  In the following discussion, we will identify a function $f$ with $\pi(f)$.  In addition, if $(A, x, y)$ is a twice-marked decorated discrete finite patch we will write $\chi_{(A, x, y)}$ in place of $\chi_{G^(A, x, y)}$.

Although the structure of $\hat{G}$ itself is essentially symmetric with respect to its first and second argument, $W^*(\hat{G})$ treats the two arguments quite differently, essentially because it is defined using left rather than right convolution.  The following corollary illustrates this.

\begin{cor}
	Suppose that $x \in G^0$, $y, y' \in G^x$, $H^* \in W^*(\hat{G})$, and $\phi(z, y) = \phi(z, y')$ for all $z \in G^x$.  Then $H \phi(x, y) = H \phi(x, y')$.
\end{cor}

\begin{proof}
	By Theorem \ref{linearoperators}, we can write $H = \pi(h)$.  Then

	\begin{eqnarray*}
	H \phi (x, y) & = & \int h(x, z) \phi(z, y) \,d\lambda_x(z) \\
	& = & \int h(x, z) \phi(z, y') \,d\lambda_x(z) \\
	& = & H \phi (x, y').
	\end{eqnarray*}
\end{proof}

In other words, operators in $W^*(\hat{G})$ preserve invariance with respect to the second argument.  (They do not preserve invariance with respect to the first argument; consider the example where $h$ is a diagonal function taking on two different values.)  We might think of operators in $W^*(\hat{G})$ as acting on the first argument only, with the second argument being a convenient way to keep track of the foliated structure.

Our next goal is to classify $W^*(\hat{G})$.

\begin{lem}
Suppose that $f \in L^2(\hat{G})$ and $g \in L^2(\hat{G_0})$.  Define $g_D(x, z) = D(x, z) g(x) = D(x, z) g(z)$.  Then
\[ f * g_D = g_D * f \]
if and only if
\[ g(x) f(x, z) = g(z) f(x, z) \]
for almost every $(x, z) \in \hat{G}$.
\end{lem}

\begin{proof}
We have
\begin{eqnarray*}
	g_D * f (x, z) & = & \int g_D (x, y) f(y, z) \,d\lambda^x(y) \\
	& = & g(x) f(x, z)
\end{eqnarray*}
and, similarly,
\[ f * g_D (x, z) = g(z) f(x, z). \]
\end{proof}

\begin{lem}\label{centerlemma}
Let $f \in L^2(\hat{G})$.  Then $\pi(f) \in Z(\hat{G})$ if and only if $f(x, y) = 0$ and $f(x, x) = f(x, y)$ for almost every $(T, x, y)$ with $x \not= y$.
\end{lem}

\begin{proof}
Suppose that $f(x, y) = 0$ and $f(x, x) = f(x, y)$ for almost every $(T, x, y)$ with $x \not= y$.  Then $f = g_D$ for some $g \in L^2(\hat{G_0})$.  For every $h \in L^2(\hat{G})$ and $(x, z) \in \hat{G}$, we have
\[ g(x) h(x, z) = g(z) h(x, z), \]
and the preceding lemma implies $h * g_D = g_D * h$; i.\,e., $f = g_D \in Z(W^*(\hat{G}))$.

Conversely, suppose that $f \in Z(\hat{G})$.  For every $g \in L^2(\hat{G_0})$, the preceding lemma gives
\[ g(x) f(x, z) = g(z) f(x, z) \]
for almost all $(x, z)$.  This in particular holds when $g = \chi_{(A, x)}$ for some discrete decorated finite patch $(A, x)$.  We can choose a countable set of finite patches $(A_k, v_k)$ such that the $\chi_{(A_k, v_k)}$ generate the $\sigma$-algebra of $\hat{G}$.  If $x \not= z$, then there must exist some $k$ such that $\chi_{(A_k, v_k)}(x) \not= \chi_{(A-k, v_k)}(z)$; otherwise $(T, x)$ would be isomorphic to $(T, z)$, and $x$ and $z$ would be two representatives of the same orbit of $\mathrm{Aut}(T)$.  Hence $f(x, z) = 0$ for almost all $(x, z)$ with $x \not= z$, and we may assume that $f = h_D$ for some $h \in L^2(\hat{G_0})$.

If $g \in L^2(\hat{G})$ is arbitrary, Lemma \ref{centerlemma} gives
\[ h(x) g(x, z) = h(z) g(x, z) \]
for almost all $(x, z)$.  By choosing a $g \in L^2(\hat{G})$ that is nowhere zero, this implies that $h(x) = h(z)$ for almost all $(x, z)$; that is, $h$ is constant on almost every leaf.
\end{proof}

\begin{cor}\label{factorergodic}
The algebra $W^*(\hat{G})$ is a factor if and only if $\nu$ is ergodic.
\end{cor}

\begin{proof}
If $\hat{G}$ is ergodic and $f \in Z(\hat{G})$, the above lemma implies that $f(x, y) = 0$ and $f(x, x) = f(x, y)$ for almost every $(T, x, y)$ with $x \not= y$; that is, $f = h_D$ for some $h \in L^2(\hat{G})$, and $h$ is constant on almost every leaf.  Since $G$ is ergodic, this means that that $h$ is constant, and $f$ is a multiple of $D$.

Conversely, if $\hat{G}$ is not ergodic, there exists a function $h \in L^2(\hat{G}_0)$ that is not constant but is constant on almost every leaf.  Then $h_D \in Z(\hat{G})$, but $h_D$ is not a multiple of $D$.
\end{proof}

\subsection{Trace}

Many of our results will involve traces on $W^*(G)$ and $W^*(\hat{G})$.  These traces are closely related to the transverse measure $\nu$.  First we will show:

\begin{thm}
	The transverse measure $\nu$ is invariant; that is, $\nu^r = \nu^s$ where $r$ denotes the range map, $s$ denotes the source map, and
	\[ \nu_r(S) = \int_{G_0} |S \cap r^{-1}(x) \,d\nu(x) \]
	and 
	\[ \nu_s(S) = \int_{G_0} |S \cap s^{-1}(x) \,d\nu(x). \]
\end{thm}

\begin{proof}
Suppose that $A$ is a finite triangulation with marked tangent vectors $p$ and $q$.  For each $x = (T, v) \in G^0$,
\[ G_0^A \cap r^{-1}(x) \]
is the number of times $(A, q)$ appears in $(T, v)$.  This is 1 if $v$ is contained in a marked patch that looks like $(A, q)$ in $T$, and 0 otherwise.  Furthermore,
\[ \nu_r(G_0^A) = \int_X |G_0^A \cap r^{-1}(x)| \,d\nu(x) \]
is simply the probability of the marked patch $(A, q)$ appearing in a randomly chosen marked triangulation in $G_0$.  Similarly, $\nu_s(G_0^A)$ is the probability of $(A, p)$ appearing in a randomly chosen marked triangulation.

Suppose that a sphere triangulation contains $k$ distinct (though not necessarily disjoint) copies of $A$.  Each copy of $A$ must contain exactly $|Aut(A)|$ copies of $p$ (that is, tangent vectors $p'$ such that $(A, p')$ is isomorphic to $(A, p)$), and exactly $|Aut(A)|$ copies of $q$.  Furthermore, no two distinct copies of $A$ can share a copy of $p$ (respectively, $q$), since knowing the location of $p$ specifies the location of all other faces of $A$.  Since the marked tangent vector is chosen uniformly, $(A, p)$ and $(A, q)$ have the same probability of appearing in the sphere triangulation, and taking limits gives
\[ \nu_r(G_0^A) = \nu_s(G_0^A). \]
Because sets of the form $G_0^A$ generate the $\sigma$-algebra of transversals, $\nu_r = \nu_s$, and $\nu$ is invariant.

As proven in Moore and Schochet \cite{moore} (Theorem 6.30), the invariance of $\nu$ implies the existence of a trace on $W^*(G)$, which we will denote $\tau$, given by
\[ \tau(H) = \int \mathrm{Tr}(H_\ell) \,d\nu(\ell), \]
where $H_\ell$ is the restriction of $H$ to the leaf $\ell$.
\end{proof}

\begin{thm}
Let $u: G \to \mathbb{R}_{\geq 0}$ such that, for each $x$, $u|_{G_x}$ is compactly supported and has total mass 1.  Then
\[ \tau(H) = \int \mathrm{Tr}(u(x, *) (H_{G_x})) \,d\nu(x), \]
independently of $u$.
\end{thm}

\begin{proof}
We note that
\begin{eqnarray*}
	\tau(H) & = & \int \mathrm{Tr}((H_{G^y})(y)) \,d\nu(y) \\
	& = & \int \int u(x, y) \mathrm{Tr}((H_{G^y})(y) \,d\lambda^y(x)) \,d\nu(y) \\
	& = & \int \int u(x, y) \mathrm{Tr}((H_{G_x})(y) \,d\lambda_x(y)) \,d\nu(x) \\
	& = & \int \int \mathrm{Tr}(u(x, y) (H_{G_x})(y) \,d\lambda_x(y)) \,d\nu(x) \\
	& = & \int \mathrm{Tr}(u(x, *) (H_{G_x})) \,d\nu(x).
\end{eqnarray*}that is, $\tau$ can also be given by $\int \mathrm{Tr}(u(x, *) (H_{G_x})) \,d\nu(x)$, independently of $u$.
\end{proof}

We can also define, in an analogous manner, a trace (also denoted $\tau$) on $W^*(\hat{G})$.  In this case, the construction can be made simpler and more explicit.

\begin{defi} [Trace on $W^*(\hat{G})$]
For $H \in W^*(\hat{G})$, let
\[ \tau(H) = \int H(\delta_x)(x) \,d\nu(x). \]
\end{defi}

\begin{thm}
	The function $\tau$ is a trace with total mass 1.
\end{thm}
\begin{proof}
	We note that
	\begin{eqnarray*}
	\tau(HJ) & = & \int HJ(\delta_x)(x) \,d\nu(x) \\
	& = & \int \int J(\delta_x)(y) H(\delta_y)(x) \,d\lambda_x(y) \,d\nu(x) \\
	& = & \int \int H(\delta_y)(x) J(\delta_x)(y) \,d\lambda_y(x) \,d\nu(y) \\
	& = & \int JH(\delta_y)(y) \,d\nu(y) \\
	& = & \tau(JH),
	\end{eqnarray*}
	and hence $\tau$ satisfies the definition of a trace.  That $\tau(1) = 1$ follows from the definition.
\end{proof}

We can now classify $W^*(\hat{G})$, subject to the conditions that we used to guarantee amenability and ergodicity.

\begin{thm}\label{big1}
Suppose that $\nu$ is such that, for a.e. $x$, $G^x$ is recurrent and $\nu^x = \nu$.  If $\nu(x) > 0$ for some (equivalently, a.e.) $x$, then $W^*(\hat{G})$ is a hyperfinite type $I_{|\hat{G}_0|}$ factor.  Otherwise, $W^*(\hat{G})$ is a hyperfinite type $II_1$ factor.
\end{thm}

\begin{proof}
Since $\hat{G}$ is amenable, $W^*(\hat{G})$ is hyperfinite by Delaroche and Renault \cite{delaroche}.  Since $\hat{G}$ is ergodic, $W^*(\hat{G})$ is a factor by corollary \ref{factorergodic}.  The type of $W^*(\hat{G})$ follows by considering the possible values of $\tau$.
\end{proof}

Next we introduce a class of operators that will be particularly well-behaved.
\begin{defi}[Finite hopping range]
Suppose that $H \in W^*(\hat{G})$ and there exists a radius $r$ such that $H_{G^x}(\delta_x)$ is supported in $B_r(x)$ for all $x$, and $B_r(x) \cong B_r(y)$ implies $H_{G^x}(\delta_x) = H_{G^y}(\delta_y)$ where we identify $B_r(x)$ with $B_r(y)$.  Then we say that $H$ has hopping range $r$.

Similarly, suppose that $H \in W^*(G)$ and there exists a radius $r > 1$ with $H_{G^x}(\delta_x)$ supported in $B_r(x)$ for all $x$.  Suppose that, for all $x$ and $y$ with $B_r(x) \cong B_r(y)$ and $f$ supported on $B_1(x)$, then $H_{G^x}(f) = H_{G^y}(f)$ (where we define $f$ on $G^y$ via the isomorphism between $G^x$ and $G^y$).  Then we say that $H$ has hopping range $r$.
\end{defi}

\begin{thm}\label{traceaverage}
If $\nu^x = \nu$ for almost all $x$ and $H \in W^*(\hat{G})$ has finite hopping range, then $\tau(H)$ is the limit of the leafwise trace averaged along a random walk on an almost arbitrary leaf; that is,
\[ \tau(H) = \lim_{n \to \infty} \mathrm{Tr}(\nu^x_n(G^A) H_{G^{x_A}}) \]
for almost all $x$.
\end{thm}

\begin{proof}
Since $H$ has finite hopping range, $H(\delta_x)(x)$ is constant on $G^A$ for $A$ a patch of radius at least $r$.  It follows that
\[ \tau(H) = \int H_{G^{x_A}}(\delta_{x_A})(x_A) \,d\nu(G^A), \]
where the integral is over all patches $A$ of radius $r$, and $x_A$ is an arbitrary element of $G^A$.  (If $\mathcal{D}$ is discrete, this integral is just a sum.)  If $x$ is chosen such that $\nu^x = \nu$, we can write
\begin{eqnarray*}
\tau(H) & = & \int H_{G^{x_A}}(\delta_{x_A})(x_A) \,d\nu^x(G^A) \\
& = & \lim_{n \to \infty} \int H_{G^{x_A}}(\delta_{x_A})(x_A) \,d\nu^x_n(G^A) \\
& = & \lim_{n \to \infty} \mathrm{Tr}(\nu^x_n(G^A) H_{G^{x_A}}),
\end{eqnarray*}
which is true for almost all $x$.
\end{proof}

This argument can also be applied to any other sequence of measures that converges to $\nu$.  In particular, we have this theorem, whose proof is analogous to Theorem \ref{traceaverage}:

\begin{thm}
	If $\nu = \lim_{n \to \infty} \nu_n$ and $H \in W^*(\hat{G})$ has finite hopping range, then
	\[ \tau(H) = \lim_{n \to \infty} \mathrm{Tr}(\nu_n(G^A) H_{G^{x_A}}). \]
\end{thm}

\subsection{Density of states}

\begin{defi}[Density of states]
For $H \in W^*(\hat{G})$ and $\phi \in C_c(\mathbb{R})$, define $\kappa_H(\phi) = \tau(\phi(H))$.  We call $\kappa_H$ the density of states of $H$.
\end{defi}

\begin{thm}\label{big2}
If $H$ has finite hopping range, and $\nu'_n$ is a sequence of measures converging to $\nu$, then we can obtain the density of states $\kappa_H(\phi)$ as a limit of the leafwise local trace $\mathrm{Tr} (\phi(H))$ averaged with respect to $\nu'_n$; that is,
\[ \kappa_H(\phi) = \lim_{n \to \infty} \mathrm{Tr} \left(\nu'_n(\hat{G}^A) \phi(H) \right). \]
\end{thm}

\begin{proof}
Given any continuous $\phi: \mathbb{R} \to \mathbb{R}$ and $\epsilon > 0$, there exists a Weierstrass polynomial $p$ such that $|\phi(t) - p(t)| < \epsilon / 3$ for $|t| < \|H\|$.  Let $\hat{\phi} = \phi - p$.  Now $p(H)$ is a polynomial of $H$ and therefore also has finite hopping range.  We can therefore write

\begin{eqnarray*}
\kappa_H(p) & = & \tau(p(H)) \\
& = & \lim_{n \to \infty} \mathrm{Tr} \left(\nu^x_n(\hat{G}^A) p(H)_{\hat{G}^{x_A}} \right).
\end{eqnarray*}
for almost all $x$.  In other words, for sufficiently large $n$ and almost all $x$, we have
\[ \left| \kappa_H(p) - \mathrm{Tr} \left(\nu'_n(\hat{G}^A) p(H)_{\hat{G}^{x_A}} \right) \right| < \frac{\epsilon}{3}. \]
Furthermore,
\begin{eqnarray*}
\left| \kappa_H(\phi) - \kappa_H(p) \right| & = & \| \tau(\phi(H)) - \tau(p(H)) \| \\
& = & | \tau(\hat{\phi}) | \\
& \leq & \frac{\epsilon}{3}
\end{eqnarray*}
and
\begin{eqnarray*}
\left| \mathrm{Tr} \left(\nu'_n(\hat{G}^A) \phi(H)_{\hat{G}^{x_A}} \right) - \mathrm{Tr} \left(\nu'_n(\hat{G}^A) p(H)_{\hat{G}^{x_A}} \right) \right| & = & \left| \mathrm{Tr} \left(\nu'_n(\hat{G}^A) \hat{\phi}(H)_{\hat{G}^{x_A}} \right) \right| \\
& \leq & \left| \mathrm{Tr} \left(\nu'_n(\hat{G}^A) \epsilon / 3 \right) \right| \\
& \leq & \frac{\epsilon}{3},  
\end{eqnarray*}
and we conclude that
\[ \left| \kappa_H(\phi) - \mathrm{Tr} \left( \nu'_n(\hat{G}^A) \phi(H)_{\hat{G}^{x_A}} \right) \right| \leq \epsilon. \]
It follows that
\[ \kappa_H(\phi) = \lim_{n \to \infty} \mathrm{Tr} \left(\nu'_n(\hat{G}^A) \phi(H)_{\hat{G}^{x_A}} \right). \]
\end{proof}

Two natural choices for the $\nu'_n$ are the measures $\nu_n$ on spheres and the measures $\nu^x_n$ obtained from a random walk on a generic leaf (assuming $G$ is ergodic).  Hence we have the following corollaries:

\begin{cor}
If $H$ has finite hopping range and $G$ is ergodic, then we can obtain $\kappa_H(\phi)$ as a limit of $\mathrm{Tr} (\phi(H))$ averaged along a random walk on an almost arbitrary leaf; that is,
\[ \kappa_H(\phi) = \lim_{n \to \infty} \mathrm{Tr} \left( \nu^x_n(\hat{G}^A) \phi(H)_{\hat{G}^{x_A}} \right) \]
for almost all $x$.
\end{cor}

\begin{cor}
If $H$ has finite hopping range, then we can obtain $\kappa_H(\phi)$ as a limit of $\mathrm{Tr} (\phi(H))$ averaged on sphere triangulations; that is,
\[ \kappa_H(\phi) = \lim_{n \to \infty} \mathrm{Tr} \left( \nu_n(\hat{G}^A) \phi(H). \right) \]
\end{cor}

The latter provides a convenient means of approximating $\kappa_H$, since $\mathrm{Tr} \phi(H)$ on a finite sphere triangulation can be calculated by computing eigenvalues.  We also wish to extend this result to some operators that do not have finite hopping range.  Recall that $D \in L^2(\hat{G})$ denotes the diagonal function.

\begin{lem}
We have
\[ \langle H D, D \rangle = \tau(H). \]
\end{lem}

\begin{proof}
We have
\begin{eqnarray*}
	\langle H D, D \rangle & = & \int \int (H D)(x, y) D(x, y) \,d\lambda_x(y) \,d\mu(x) \\
	& = & \int (H D)(x, x) \,d\mu(x) \\
	& = & \int H(\delta_x)(x) \,d\mu(x),
\end{eqnarray*}
which is precisely our definition of $\tau(H)$ on $\hat{G}$.
\end{proof}

\begin{cor}
Suppose that $H_n, H \in W^*(\hat{G})$ such that $H_n \to H$ weakly; that is,
\[ \langle H_n x, y \rangle = \langle H x, y \rangle \]
for all $x, y \in L^2(\hat{G})$.  Then $\tau(H_n) \to \tau(H)$.
\end{cor}

\begin{lem}
Suppose that $G$ is ergodic; that $H_n, H \in W^*(\hat{G})$ such that $H_n \to H$ in norm; and that $\phi \in C_c(\mathbb{R})$.  Then
\[ \kappa_{H_n}(\phi) \to \kappa_H(\phi). \]
\end{lem}

\begin{proof}
We first claim that $\phi(H_n) \to \phi(H)$ in norm.  Suppose that $\epsilon > 0$.  The hypotheses imply $H_n$ is bounded, so there exists $M$ with $|H_n| < M$ for all $n$ and $|H| < M$.  We can choose a Weierstrass polynomial $p$ with $|p - \phi|_\infty < \epsilon / 3M$.  Furthermore, since $H_n \to H$, for sufficiently large $n$ we have $|p(H_n) - p(H)| < \epsilon / 3$, and 
\begin{eqnarray*}
	|\phi(H_n) - \phi(H)| & \leq & |\phi(H_n) - p(H_n)| + |p(H_n) - p(H)| + |p(H) - \phi(H)| \\
	& \leq & \epsilon.
\end{eqnarray*}
Since norm convergence is stronger than weak convergence, it follows from the above corollary and the definition of $\kappa$ that $\kappa_{H_n}(\phi) \to \kappa_H(\phi)$.
\end{proof}

\begin{cor}
Let $H_m$ and $H$ be operators such that $H_m$ has finite hopping range for each $m$, $H_m \to H$ in norm, and $\nu'_n \to \nu$.  Then
\[ \kappa_{H}(\phi) = \lim_{m \to \infty} \lim_{n \to \infty} \mathrm{Tr} \left(\nu'_n(\hat{G}^A) \phi(H_m)_{\hat{G}^{x_A}} \right). \]
\end{cor}

\subsection{Jumps of the IDS}

One of the main results of Lenz and Veseli\'{c} \cite{lenz2007} relates compactly supported eigenfunctions to jumps of the IDS.  We now have the framework to prove a similar result in our setting.

\begin{thm}\label{big3}
Suppose that $\mathcal{D}$ is discrete, $\hat{G}$ is ergodic, and that $H \in W^*(\hat{G})$ is $G$-invariant and has finite hopping range.  Suppose also that $t \in \mathbb{R}$.  The following are equivalent:

\begin{enumerate}
	\item The density of states $\kappa_H(t) > 0$.
	\item For some $x$, $H|_{\hat{G}^x}$ has an eigenfunction with eigenvalue $t$ supported on some finite patch $A$ with $\nu(G^A) > 0$. 
	\item For almost all $x$, $\ker (H_{\hat{G}^x} - \lambda I)$ is nontrivial and spanned by compactly supported eigenfunctions.
\end{enumerate}
\end{thm}

\begin{proof}
The implication $(3) \Rightarrow (2)$ is trivial.  To prove $(2) \Rightarrow (1)$, suppose that (2) holds.  Let $f$ be a compactly supported eigenfunction on $G^x$ with eigenvalue $t$.  Since $H$ has finite hopping range, there exists a radius $r$ such that $Hf|_{\hat{G}^x} = Hf|_{\hat{G}^y}$ for any $y \in \hat{G}_0^{B_r(x)}$.

Define $u(x, y) = \frac{\chi_{B_r(x)}(y)}{|B_r(x)|}$.  For each $y \in \hat{G}_0^{B_r(x)}$, $H$ has $f$ as an eigenfunction, and hence $\mathrm{Tr}(H|_{A_y}) > 0$.  Because $\hat{G}_0^{B_r(x)}$ has positive measure, (1) follows.  It only remains to prove $(1) \Rightarrow (3)$.

For $n \in \mathbb{N}$, let $A_n$ denote the $n$th interior of $A$; that is, the set of $x \in A$ with $B_n(x) \subset A$.

\begin{lem}
Let $P \in W^*(\hat{G})$ with $P > 0$ and let $R > 0$.  Then there exists a finite patch $A$ such that $\mathrm{ran}(\chi_A P|_{\hat{G}^A}) \cap \ell^2(\hat{G}_0^{A_R}) \not= \{0\}$.
\end{lem}

\begin{proof}
Since $P > 0$, $\tau(P) > 0$.  Because $P$ is bounded, we may assume without loss of generality that $|P| = 1$.  By amenability and ergodicity, for almost every $x$, $\hat{G}^x$ has a F{\o}lner exhaustion $\Lambda_n$, and $\lim_{n \to \infty} \frac{\mathrm{Tr}(\chi_{\Lambda_n} P|_{\Lambda_n})}{|\Lambda_n|} = \tau(P)$.  Choose $x$ such that this is true.  For $n$ sufficiently large,
\[ \mathrm{Tr}(\chi_{\Lambda_n} P|_{\Lambda_n}) \geq \frac{1}{2} \tau(P) |\Lambda_n| \]
and, since $\Lambda_n$ is a F{\o}lner exhaustion,
\[ |\Lambda_n \setminus \Lambda_{n, R}| < \frac{1}{2} \tau(P) |\Lambda_n|. \]
This implies
\begin{eqnarray*}
|\Lambda_n \setminus \Lambda_{n, R}| & < & \mathrm{Tr}(\chi_{\Lambda_n} P|_{\Lambda_n}) \\
& \leq & \dim( \mathrm{ran}(\chi_{\Lambda_n} P|_{\Lambda_n}));
\end{eqnarray*}
that is,
\begin{eqnarray*}
\dim( \mathrm{ran}(\chi_{\Lambda_n} P_|{\Lambda_n})) & \geq & \|\Lambda_n\| - \|\Lambda_{n, R}\| \\
& = & \dim(\ell^2(\Lambda_n)) - \dim(\ell^2(\Lambda_{n, R})),
\end{eqnarray*}
implying that
\[ \mathrm{ran}(\chi_{\Lambda_n} P|_{\Lambda_n}) \cap \ell^2(\Lambda_{n, R}) \not= \{0\}. \]
\end{proof}

\begin{lem}
Let $H \in W^*(\hat{G})$ have finite hopping range and let $\lambda \in \mathbb{R}$.  Define $P = E_H(\{\lambda\})$ and let $P_c$ be the projection of $P$ onto the subspace generated by compactly supported eigenfunctions of $H$.  Then $P = P_c$.
\end{lem}

\begin{proof}
Suppose not.  Define $Q = P - P_c$.  Then $Q > 0$.  Define $R$ to be twice the hopping range of $H$ and apply the previous lemma, so that $A$ is a finite patch and we can choose a nonzero $f \in \ell^2(\hat{G}_0^A)$ with $\chi_A Q(f) \in \ell^2(\hat{G}_0^{A_R}) \setminus \{0\}$.  We note that
\begin{eqnarray*}
\lambda \chi_A Q(f) & = & \chi_A \lambda Q(f) \\
& = & \chi_A H Q(f) \\
& = & \chi_A (H \chi_A Q(f)) + \chi_A (H \chi_{A^c} Q(f)).
\end{eqnarray*}
Since $\chi_A Q(f) \in \ell^2(\hat{G}_0^{A_R})$, $H \chi_A Q(f) \in \ell^2(\hat{G}_0^{A_{R/2}})$ since $H$ has hopping range $R/2$.  Again using the hopping range of $H$, $H \chi_{A^c} Q(f) \in \ell^2(\hat{G}_0^{A_{R/2}^c})$.  But $\lambda \chi_A Q(f) \in \ell^2(\hat{G}_0^{A_R})$, and hence the second term is 0, giving us
\begin{eqnarray*}
\lambda \chi_A Q(f) & = & \chi_A (H \chi_A Q(f)) \\
& = & H \chi_A Q(f).
\end{eqnarray*}
This means that $\chi_A Q(f)$ is a compactly supported eigenfunction of $H$, contradicting the definition of $Q$.
\end{proof}

We can now prove $(1) \Rightarrow (3)$.  If (1) is true, then $E_H(\{\lambda\})$ is nontrivial, and hence $\ker (H_{\hat{G}^x} - \lambda I)$ is nontrivial.  The previous lemma implies that $\ker (H_{\hat{G}^x} - \lambda I)$ is spanned by compactly supported eigenfunctions.
\end{proof}

\subsection{Approximate eigenfunctions}

Our goal is to study some properties of operators on $\hat{G}$ by examining them on sphere triangulations.  One way to do so is by relating a property to another property that can be described in terms of finite patches.

\begin{defi}	[Spectrally localized function]
Let $H \in W^*(\hat{G})$, $x \in G^0$, and $f \in L^2(\hat{G}^x)$.  Let $\phi: \mathbb{R} \to [0, 1]$ and $\xi, \delta \in \mathbb{R}$.  We say that $f$ is $(H, \xi, \delta, \epsilon)$-spectrally localized if
\[ \| (1 - \chi_{[\xi - \delta, \xi + \delta]})H(f) \leq \epsilon \|f\|. \| \]
\end{defi}

\begin{defi} [Approximate eigenfunction]
Let $H \in W^*(\hat{G})$ and let $\xi \in \mathbb{C}$.  Let $f \in L^2(G^x) \setminus \{0\}$.  If $\|(Hf - \xi f)\|_2 < \zeta \|f\|_2$, we say that $f$ is a $\zeta$-approximate eigenfunction for $H$ with eigenvalue $\xi$ supported on $A_r$.
\end{defi}

\begin{lem}
Let $H \in W^*{\hat{G}}$ be self-adjoint and let $x \in G^0$ and $\xi \in \mathbb{R}$.  Let $\delta > 0$.  Suppose that $f \in L^2(G^x)$ such that $\|f\|_{L^2(G^x)} = 1$. Then:
\begin{enumerate}
	\item If $f$ is an $\epsilon \delta$-approximate eigenfunction, then $f$ is $(H, \xi, \delta, \epsilon)$-spectrally localized.
	\item If $f$ is $(H, \xi, \delta, \epsilon)$-spectrally localized, then $f$ is a $\delta + \|H\| \epsilon$-approximate eigenfunction.
\end{enumerate}  
\end{lem}

\begin{proof}
(1) Assume without loss of generality that $\xi = 0$.  Because $(1 - \chi_{[\xi - \delta, \xi + \delta]})(t) \leq t / \delta$ for all $t$, we have
\begin{eqnarray*}
\| (1 - \chi_{[\xi - \delta, \xi + \delta]}) H\|_{G^x}(f) & \leq & \|Hf\| / \delta \\
& < & \epsilon / \delta.	
\end{eqnarray*}

(2) Again assume without loss of generality that $\xi = 0$.  Because $(1 - \chi_{[\xi - \delta, \xi + \delta]})(t) \leq t / \delta$ for all $t$, we have
\begin{eqnarray*}
\|Hf\| & = & \| H (\phi H(f) + (1 - \chi_{[\xi - \delta, \xi + \delta]}) H(f))\| \\
& \leq & \|H \phi H(f)\| + \|H (1 - \chi_{[\xi - \delta, \xi + \delta]}) H(f) \| \\
& \leq & \delta + \|H\| \epsilon.
\end{eqnarray*}
\end{proof}

\begin{cor}
Let $H \in W^*{\hat{G}}$ be self-adjoint with finite hopping range $r$ and let $\xi \in \mathbb{R}$.  Let $R > r$ and $\delta > 0$.  Then

\begin{eqnarray*}
& & P(f|_{G^x} \mbox{ has a } B_{R - r}- \mbox{supported } \epsilon \delta\mbox{- approximate eigenfunction}) \\
& \leq & P(f|_{G^x} \mbox{ has a } B_{R - r}- \mbox{supported } (H, \xi, \delta, \epsilon)-\mbox{ spectrally localized function}) \\
& \leq & P(f|_{G^x} \mbox{ has a } B_{R - r}- \mbox{supported } \delta + \|H\| \epsilon\mbox{ -approximate eigenfunction}).
\end{eqnarray*}
\end{cor}

Since the left and right sides of the inequality depend only on the patch densities of patches of radius $R$, this gives a means by which we might estimate the frequency of compactly supported, spectrally localized functions.  We might also compute limits as, for example, $R \to \infty$, $\delta = C R^u$, and $\epsilon = D R^v$ for some $C, D > 0$ and $u, v < 0$.

\begin{cor}
Let $u, v \leq 0$ and $C, D > 0$.  Then
\begin{eqnarray*}
	& & \lim_{R \to \infty} P(f|_{G^x} \mbox{ has a } B_{R - r}- \mbox{supported } CD R^{u+v} \mbox{- approximate eigenfunction}) \\
	& \leq & \lim_{R \to \infty} P(f|_{G^x} \mbox{ has a } B_{R - r}- \mbox{supported } (H, \xi, \delta, D R^v)-\mbox{ spectrally localized function}) \\
	& \leq & \lim_{R \to \infty} P(f|_{G^x} \mbox{ has a } B_{R - r}- \mbox{supported } C R^u + \|H\| D R^v\mbox{ -approximate eigenfunction}).
\end{eqnarray*}
\end{cor}

Let $A$ be a decorated discrete finite patch.  Suppose that $f$ is an $\eta$-approximate eigenfunction of $H$ with eigenvalue $\xi$ supported on $A_r$.  Let $\delta, \iota > 0$ and $\phi = \chi[\xi - \delta, \xi + \delta]$.

\begin{cor}
Suppose that $B$ is an exact decorated discrete finite patch containing $k$ disjoint copies of $A$.  On each copy of $A$ we have a copy of $f$, which we can denote $f_1, \ldots, f_k$.  Then $f_i$ is $(H, \phi, \eta / \delta)$-spectrally localized, and
\[ \langle f_i, \phi H(f_i) \rangle \geq 1 - \frac{\eta}{\delta}. \]
\end{cor}

\section{Directionally invariant functions and the discrete Laplacian}

Motivated by physics, we are interested in studying properties of the Laplacian and related operators.  In order to work with the discrete spaces $\hat{G}$ and $\hat{G}_0$, we need to consider discrete analogues to these operators. Fabila Carrasco et al.\,\cite{fabila} constructed a discrete magnetic Laplacian on graphs; we will adapt this construction to our setting.

Our choice of using tangent vectors to construct $\hat{G}$ may seem inconvenient when constructing operators.  However, the following mechanism allows us to consider operators on functions that essentially ignore the directions of tangent vectors.

\begin{defi} [Directionally invariant function]
A function $f \in L^p(\hat{G})$ is directionally invariant if $f(x_1, x_2, y_1, y_2) = f(x_1, x_3, y_1, y_3)$ for all $x_k, y_k$ such that $(x_1, x_2, y_1, y_2), (x_1, x_3, y_1, y_3) \in \hat{G}$.  We denote the subspace of directionally invariant functions by $L^p_{DI}(\hat{G})$.
\end{defi}

\begin{lem}
If $f$ and $g$ are directionally invariant functions such that $f * g$ is defined, then $f * g$ is directionally invariant.
\end{lem}

\begin{proof}
We have
\begin{eqnarray*}
	f * g (x_1, x_2, z_1, z_2) & = & \int f(x_1, x_2, y_1, y_2) g(y_1, y_2, z_1, z_2) \,d\lambda_{(x_1, x_2)}(y_1, y_2) \\
	& = & \int f(x_1, x_3, y_1, y_2) g(y_1, y_2, z_1, z_3) \,d\lambda_{(x_1, x_3)}(y_1, y_2) \\
	& = & f * g (x_1, x_3, z_1, z_3).
\end{eqnarray*}
\end{proof}

\begin{defi}[Directional symmetrization of a function]
	Let $f \in L^2(\hat{G})$.  Define $N(x)$ to be the set of vertices neighboring a vertex $x$.  The directional symmetrization of $f$ is
	\[ \tilde{f}(x_1, x_2, y_1, y_2) = \frac{1}{|N(x_1)| |N(y_1)|} \sum_{\stackrel{x_3 \in N(x_1)}{y_3 \in N(y_1)}} f(x_1, x_3, y_1, y_3). \]
	By construction, $\tilde{f}$ is directionally invariant, and $\tilde{f} = f$ when $f$ is directionally invariant.
\end{defi}

\begin{defi}[Extension of an operator on directionally invariant functions]
	Let $H \in B(L^2_{DI}(\hat{G}))$.  Define
	\[ \tilde{H}(f) = H(\tilde{f}). \]
	Then $\tilde{H}$ is an operator on $L^2(\hat{G})$ that agrees with $H$ on $L^2_{DI}(\hat{G})$ (in particular, eigenfunctions of $H$ are also eigenfunctions of $\tilde{H}$).
\end{defi}

We will construct the discrete Laplacian as an operator on directionally invariant functions.

\begin{defi}[The discrete Laplacian with a potential and magnetic field]
	Let $\mathcal{D} = \mathbb{R}^4$ and denote the components $D(x_1, x_2) = (w(x_1), \bar{w}(x_1, x_2), V(x_1), \alpha(x_1, x_2))$.  Furthermore, choose $\nu$ such that, for almost all $(x_1, x_2)$, $w$ and $V$ depend only on $x_1$, $\bar{w}$ is symmetric, and $\alpha$ is antisymmetric.
	
We think of $w$ as the vertex weight and $\bar{w}$ as the edge weight.  We could construct $\nu$ so that these weights are chosen either deterministically based on the surrounding geometry (one natural choice is $w(x_1) = \deg(x_1)$ and $\bar{w}(x_1, x_2) = 1$) or randomly.
We define
\[ \Delta_{\mathrm{disc}, V, B} f(x, z) = \frac{1}{w(x_1)} \sum_{y_1 \in N(x_1)} \bar{w}(x_1, y_1) \left( e^{i \alpha(x_1, y_1)} f(y_1, z_1) \right) + (V(x_1) - 1) f(x_1, z_1). \]
\end{defi}

When $\alpha$ and $V$ are everywhere 0, this reduces to 
\[ \Delta_{\mathrm{disc}} f(x_1, z_1) = \frac{1}{w(x_1)} \sum_{y_1 \in N(x_1)} \bar{w}(x_1, y_1) \left( f(y_1, z_1) - f(x_1, z_1) \right). \]

In the case where $\alpha = V = 0$, $w(x) = \deg(x)$, and $\bar{w} = 1$ everywhere, Figure \ref{eigen} shows an eigenfunction of $\Delta_{\mathrm{disc}}$.

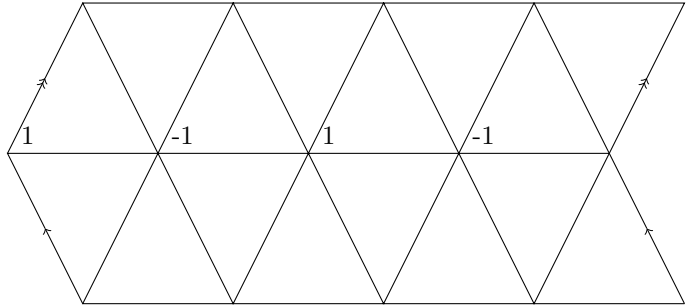
\begin{figure}
\centering
\begin{tikzpicture}
\draw [->] (1, 0) -- (0.5, 1);
\draw (0.5, 1) -- (0, 2) -- (2, 2) -- (1, 0) -- (3, 0) -- (2, 2) -- (4, 2) -- (3, 0) -- (5, 0) -- (4, 2) -- (6, 2) -- (5, 0) -- (7, 0) -- (6, 2) -- (8, 2) -- (7, 0) -- (9, 0);
\draw [->] (9, 0) -- (8.5, 1);
\draw (8.5, 1) -- (8, 2) -- (7, 4) -- (9, 4) -- (8.5, 3);
\draw [->>] (8, 2) -- (8.5, 3);
\draw [->>] (0, 2) -- (0.5, 3);
\draw (0.5, 3) -- (1, 4) -- (2, 2) -- (3, 4) -- (4, 2) -- (5, 4) -- (6, 2) -- (7, 4) -- (1, 4);
\node[above right] at (0, 2) {\,1};
\node[above right] at (2, 2) {\,-1};
\node[above right] at (4, 2) {\,1};
\node[above right] at (6, 2) {\,-1};
\end{tikzpicture}
\caption{Example of a compactly supported eigenfunction of $\Delta_{\mathrm{disc}}$ with eigenvalue $-4/3$.  The function takes the value 0 on all unlabeled vertices.} \label{eigen}
\end{figure}

\section{Questions}

\begin{itemize}
\item Our amenability result used the result of Benjamini and Schramm \cite{benjamini} which requires a bound on the degrees of vertices.  If we allow vertices to have unbounded degrees, can we still ensure amenability?
\item Do any of our results hold in higher dimensions?  (Again, Benjamini and Schramm \cite{benjamini} would not apply.)
\item Do Conjectures \ref{thisconjecture}, \ref{thatconjecture}, and \ref{theotherconjecture} hold?
\item In the case of a measure $\nu$ derived from a generalized substitution tiling, can we prove results similar to those that have been proven for substitution tilings?
\item Can we determine any stronger results about ``convergence'' of eigenvectors; that is, results about properties of eigenvectors or near-eigenvectors of an operator on $G$ by examining analogous operators on spheres?
\item Can we say anything specific about $\Delta_{\mathrm{disc}}$ or any other interesting operator for any particular choice of $\nu$?  In particular, can we prove any results analogous to those of Fabila Carrasco et al.\,\cite{fabila}?
\end{itemize}

\section{Acknowledgements}

This paper is based on a dissertation written under the supervision of Eric Babson.  The author is grateful to him and to Jerry Kaminker for their guidance and support.

\bibliographystyle{jncg}

\bibliography{triangulations8}

\end{document}